\documentclass[A4,12pt]{article}

\usepackage{amssymb}
\usepackage{amsmath}
\usepackage{amsthm}
\usepackage{graphicx}
\usepackage{color}
\usepackage{float}
\usepackage{nicefrac}
\usepackage{epstopdf}
\usepackage{amsfonts}
\usepackage{subfig}
\usepackage{graphicx}
\usepackage{tikz} 
\usepackage{pgfplots}
\usepackage{algorithm}
\usepackage{algpseudocode}

  \usepackage{multirow}
  
  \newcommand{\R}[1]{\mathbb{R}^{#1}}

  \newcommand{\CX}{\mathcal{X}}
  \newcommand{\CY}{\mathcal{Y}}
  \newcommand{\CZ}{\mathcal{Z}}
  \newcommand{\dif}{\mathrm{d}}
  
  \DeclareMathOperator*{\argmin}{\arg \min}%
  \DeclareMathOperator*{\argmax}{\arg \max}%

 \newcommand*\Let[2]{\State #1 $\gets$ #2}

\newtheorem{lemma}{Lemma}[section]
\newtheorem{theorem}{Theorem}[section]

\title{A Variational Model for Joint Motion Estimation and Image Reconstruction}

\author{Martin Burger\thanks{Institute for Computational and Applied Mathematics and Cells in Motion Cluster of Excellence, University of M{\"u}nster, Orl\'{e}ans-Ring 10, 48149 M{\"u}nster, Germany, Email: \href{mailto:martin.burger@wwu.de}{\texttt{martin.burger@wwu.de}}}, Hendrik Dirks\thanks{Institute for Computational and Applied Mathematics and Cells in Motion Cluster of Excellence, University of M{\"u}nster, Orl\'{e}ans-Ring 10, 48149 M{\"u}nster, Germany, Email: \href{mailto:hendrik.dirks@wwu.de}{\texttt{hendrik.dirks@wwu.de}}}, Carola-Bibiane Sch{\"o}nlieb\thanks{Department of Applied Mathematics and Theoretical Physics (DAMTP), University of Cambridge, Wilberforce Road, Cambridge CB3 0WA, United Kingdom, Email: \href{mailto:C.B.Schoenlieb@damtp.cam.ac.uk}{\texttt{C.B.Schoenlieb@damtp.cam.ac.uk}}}}

\begin{document}

\maketitle

\begin{abstract}
The aim of this paper is to derive and analyze a variational model for the joint estimation of motion and reconstruction of image sequences, which is based on a time-continuous Eulerian motion model. The model can be set up in terms of the continuity equation or the brightness constancy equation.
The analysis in this paper focuses on the latter for robust motion estimation on sequences of two-dimensional images. We rigorously prove the existence of a minimizer in a suitable function space setting. Moreover, we discuss the numerical solution of the model based on primal-dual algorithms and investigate several examples. Finally, the benefits of our model compared to existing techniques, such as sequential image reconstruction and motion estimation, are shown.
\end{abstract}



\section{Introduction}
Image reconstruction and motion estimation are important tasks in image processing. Such problems arise for example in modern medicine, biology, chemistry or physics, where even the smallest objects are observed by high resolution microscopes. To characterize the dynamics involved in such data, velocity fields between consecutive image frames are calculated. This is challenging, since the recorded images often suffer from low resolution, low contrast, different gray levels and noise. Methods that simultaneously denoise the recorded image sequence and calculate the underlying velocity field offer new opportunities, since both tasks may endorse each other. \\
Our ansatz aims at reconstructing a given sequence $u$ of images and calculating flow fields $\boldsymbol{v}$ between subsequent images at the same time. For given measurements $f=Ku$ this can be achieved by minimizing the variational model
\begin{align}
\label{equation:generalModelEQ}
	&\int_0^T \frac{1}{2}\left\|K_tu(\cdot,t)-f(\cdot,t)\right\|_2^2 + \alpha \mathcal{R}(u(\cdot,t))  + \beta \mathcal{S}(\boldsymbol{v}(\cdot,t)) dt\\
	\text{s.t.} &\quad \mathcal{M}(u,\boldsymbol{v}) = 0\notag
\end{align}
with respect to $u$ and $\boldsymbol{v}$ simultaneously. The denoising part is based on the ROF model \cite{rudin1992nonlinear}. The first part $\left\|Ku-f\right\|_2^2$ connects the input data $f$ with the image sequence $u$ via a linear operator $K$. Depending on the application $K$ may model the cutting out of a subset $\Sigma\subset\Omega$ for inpainting, a subsampling for super resolution, a blur for deconvolution or a Radon transform for computed tomography. Additional a-priori information about the structure of $u$ respectively $\boldsymbol{v}$ can be incorporated into each frame via the regularization terms $\mathcal{R}(u(\cdot,t))$ and $\mathcal{S}(\boldsymbol{v}(\cdot,t))$, while their significance is weighted using $\alpha$ and $\beta$. Finally, flow field and images are coupled by a constraint $\mathcal{M}(u,\boldsymbol{v}) = 0$ (e.g. the optical flow \eqref{subsection:noiseSensitivity}).\\
In the last two decades, variational models for image reconstruction have become very popular. One of the most famous models, introduced by Rudin, Osher and Fatemi in 1992 \cite{rudin1992nonlinear}, is the total variation (TV) model, where the authors couple a L$^2$ data fidelity term with a total variation regularization. Data-term and regularizer in the ROF model match with the first two terms model \eqref{equation:generalModelEQ}. The TV-regularization results in a denoised image with cartoon-like features. This model has also been adapted to image deblurring \cite{Wang07afast}, inpainting \cite{shen2002mathematical} and superresolution \cite{mitzel2009video,unger2010convex} and tomographic reconstruction \cite{sawatzky2008accurate,kosters2011emrecon}. We collectively call these image reconstruction models.\\
Estimating the flow from image sequences has been discussed in the literature for decades. Already in 1981, Horn and Schunck proposed a variational model for flow estimation \cite{horn1981determining}. This basic model uses the $L^2$ norm for the optical flow term as well as for the gradient regularizer and became very popular. Aubert et al. analyzed the $L^1$ norm for the optical flow constraint \cite{aubert1999computing} in 1999 and demonstrated its advantages towards a quadratic $L^2$ norm. In 2006, Papenberg, Weickert et al. \cite{papenberg2006highly} introduced the total variation regularization, respectively the differentiable approximation, to the field of flow estimation. An efficient duality-based $L^1-TV$ algorithm for flow estimation was proposed by Zach, Pock and Bischof in 2007 \cite{zach2007duality}. Model \eqref{equation:generalModelEQ} also incorporates a flow estimation problem by the constraint $\mathcal{M}(u,\boldsymbol{v})$ and suitable regularization $\mathcal{S}(\boldsymbol{v}(\cdot,t))$. \\
The topic of joint models for motion estimation and image reconstruction was already discussed by Tomasi and Kanade \cite{tomasi1992shape} in 1992. Instead of a variational approach, they used a matrix-based discrete formulation with constraints to the matrix rank to find a proper solution. In 2002, Gilland, Mair, Bowsher and Jaszczak published a joint variational model for gated cardiac CT \cite{gilland2002simultaneous}. For two images, they formulate a data term, based on the Kullback-Leibler divergence (cf. \cite{brune20104d} for details) and incorporate the motion field via quadratic deformation term and regularizer. 
In the field of optimal control Borzi, Ito and Kunisch \cite{borzi2003optimal} formulated a smooth cost functional for an optimal control problem that incorporates the optical flow formulation with unknown image sequence and motion field with additional initial value problem for the image sequence.\\
Bar, Berkels, Rumpf and Sapiro proposed a variational framework for joint motion estimation and image deblurring in 2007 \cite{bar2007variational}. The underlying flow is assumed to be a translation and coupled into a blurring model for the foreground and background. This results in a Mumford-Shah-type functional. Also in 2007, Shen, Zhang, Huang and Li proposed a statistical approach for joint motion estimation, segmentation and superresolution \cite{shen2007map}. The model assumes an affine linear transformation of the segmentation labels to incorporate the dynamics and is solved calculating the MAP solution. Another possible approach was given by Brune in 2010 \cite{brune20104d}. The 4d (3d + time) variational model consists of an $L^2$ data term for image reconstruction and incorporates the underlying dynamics using a variational term, introduced by Benamou and Brenier \cite{benamou2000computational,benamou2002monge}. In our model, the constraint $\mathcal{M}(u,\boldsymbol{v})$ connects image sequence $u$ and velocity field $\boldsymbol{v}$.
We mention recent development in \cite{suhr}, which also discusses a joint motion estimation and image reconstruction model in a similar spirit. The focus there is however motion compensation in the reconstruction relative to an initial state, consequently a Lagrangian approach with the initial state as reference image is used and the motion is modeled via hyperelastic deformations. Finally, in \cite{benamou:hal-01295299} Benamou, Carlier, Santambrogio draw a connection to stochastic Mean Field Games, where the underlying motion is described from the Eulerian and Lagrangian perspective.
\subsection{Contents}
The paper is structured as follows: In Section \ref{section:jointModel} we shortly introduce a basic framework for variational image reconstruction and motion estimation and afterwards combine both which leads to our joint model. Afterwards, we give a detailed proof for the existence of a minimizer based on the fundamental theorem of optimization in Section \ref{section:analyticalResults}. Finally, we introduce a numerical framework  for minimizing our model in Section \ref{section:numericalFramework} and provide applications to different image processing applications in Section \ref{section:results}.

\section{Joint motion estimation and image reconstruction}
\label{section:jointModel}
\subsection{Noise sensitivity of motion estimation}
\label{subsection:noiseSensitivity}
One of the most common techniques to formally link intensity variations in image sequences $u(x,t)$ to the underlying velocity field  $\boldsymbol{v}(x,t)$ is the optical flow constraint. Based on the assumption that the image intensity $u(x,t)$ is constant along a trajectory $x(t)$ with $\frac{dx}{dt}=\boldsymbol{v}(x,t)$ we get using the chain-rule
\begin{align}
	0 = \frac{du}{dt} =  \frac{\partial u}{\partial t} + \sum_{i=1}^n \frac{\partial u}{\partial x_i}\frac{dx_i}{dt} = u_t + \nabla u\cdot\boldsymbol{v}.
	\label{equation:opticalFlowConstraint}
\end{align}
The last equation is generally known as the \textbf{optical flow constraint}. The constraint constitutes in every point $x\in\Omega_T$ one equation, but in the context of motion estimation from images we usually have two or three spatial dimensions. Consequently, the problem is massively underdetermined. However, it is possible to estimate the motion using a variational model  
\begin{align*}
\min_{\boldsymbol{v}} \mathcal{D}(u,\boldsymbol{v}) + \alpha \mathcal{R}(\boldsymbol{v}),
\end{align*}
where $\mathcal{D}(u,\boldsymbol{v})$ represents the so-called \textit{data term} and incorporates the optical flow constraint in a suitable norm. The second part $\mathcal{R}(\boldsymbol{v})$ models additional a-priori knowledge on $\boldsymbol{v}$ and is denoted as \textit{regularizer}. The parameter $\alpha$ regulates between data term and regularizer.\\
Possible choices for the data term are 
\begin{align*}
\mathcal{D}(u,\boldsymbol{v}) := \frac{1}{2}\left\|\boldsymbol{v}\cdot\nabla u + u_t\right\|_2^2,\text{ or }\mathcal{D}(u,\boldsymbol{v}) := \left\|\boldsymbol{v}\cdot\nabla u + u_t\right\|_1.
\end{align*}	
The quadratic L$^2$ norm can be interpreted as solving the optical flow constraint in a least-squares sense inside the image domain $\Omega$. On the other hand, taking the L$^1$ norm enforces the optical flow constraint linearly and is able to handle outliers more robust \cite{aubert1999computing}. \\
The regularizer $\mathcal{R}(\boldsymbol{v})$ has to be chosen such that the a-priori knowledge is modeled in a reasonable way. If the solution is expected to be smooth, a quadratic L$^2$ norm on the gradient of $\boldsymbol{v}$ is chosen
as in the classical Horn-Schunck model.
Another possible approach is to choose the \textit{total variation} (TV) of $\boldsymbol{v}$ if we expect piecewise constant parts of motion, an approach we merely pursue in this paper. 

\begin{figure}[H]
	\centering
	\includegraphics[width=0.4\linewidth]{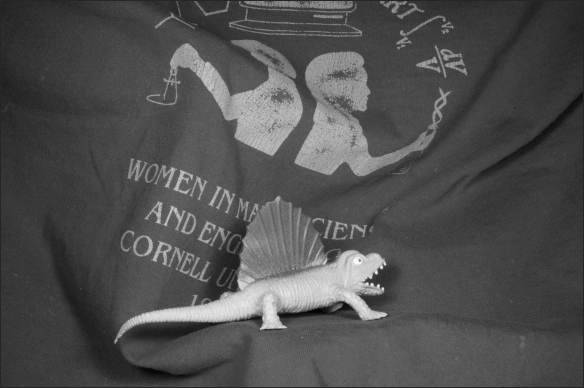}\quad 
	\includegraphics[width=0.4\linewidth]{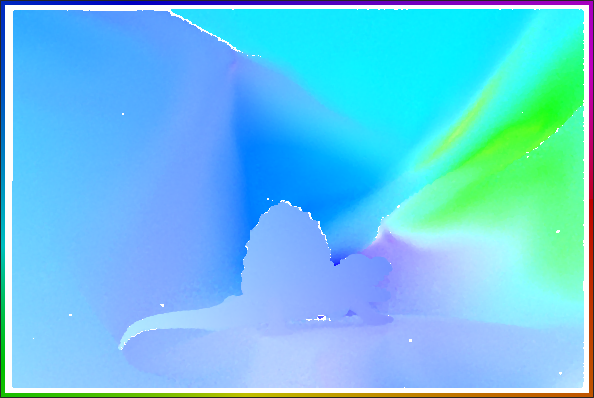}
	\caption{Image and color-coded ground-truth velocity field from the Middlebury optical flow database \cite{baker2011database}.}
	\label{fig:jointMotivation}
\end{figure}
In practical applications (e.g. microscopy) the recorded images often come with a lack of image quality which is caused by low acquisition times. This leads to another very interesting aspect in motion estimation - how does the noise-level on the image data correspond to the quality of the estimated velocity field $\boldsymbol{v}$. To answer this question we created a series of noisy images, where Gaussian noise with increasing variance $\sigma$ was added. Compare Figure \ref{fig:jointMotivation} for one of the two images and the corresponding ground-truth velocity field $\boldsymbol{v}_{GT}$. Afterwards, we estimated the motion using the $L^1-TV$ optical flow algorithm. In Figure \ref{fig:plotAEEvsNoise} we plotted the variance of noise on the x-axis versus the \textbf{absolute endpoint error} (see Equation \eqref{definition:aee}) of the reconstruction on the y-axis.
We observe that already small levels of noise have massive influence to the motion estimation process. Consequently, before estimating the motion field, a preprocessing step may be applied to remove the noise. A more advanced technique is a variational model that is able to simultaneously denoise images and estimate the underlying motion, while both tasks improve each other.

\begin{figure}[H]
	\centering\hspace{-10mm}
	\includegraphics[width=0.4\linewidth]{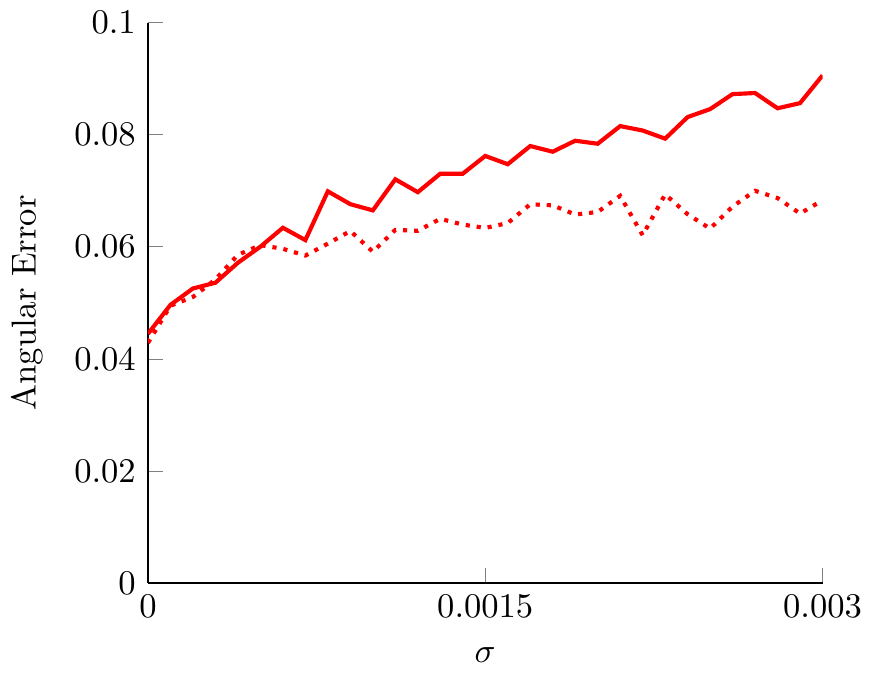}\quad 
	\includegraphics[width=0.4\linewidth]{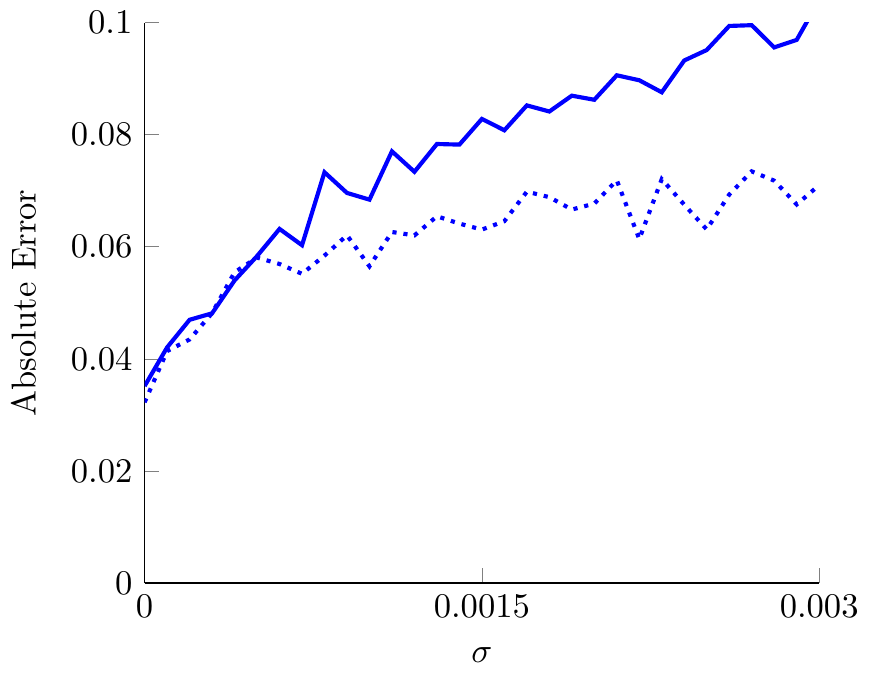}
	\caption{Absolute endpoint error (left) and angular error (right) for L$^1$-TV estimated velocity field (solid) and joint TV-TV optical flow model (dotted) for increasing levels of noise in the underlying data.}
	\label{fig:plotAEEvsNoise}
\end{figure}
\subsection{Proposed model}
In the reconstruction process we deal with measured data $f$, which can be modeled as $f = Ku + \delta$, where $\delta$ represents additive noise, often assumed to be Gaussian. The linear operator $K$ represents the forward operator modeling the relation of the image sequence $u$ on the measured data $f$. This general choice allows us to model applications such as denoising, deblurring, inpainting, superresolution, or even a Radon transform (dynamic CT, PET). Simultaneously we seek for the velocity field $\boldsymbol{v}:\Omega_T\rightarrow\R{2}$ describing the motion in the underlying image sequence $u$. \\
To reconstruct both, $u$ and $\boldsymbol{v}$ at the same time, we propose the following general model
\begin{align}
\label{model:generalModelRecMoEst}
\argmin_{u,\boldsymbol{v}} &\int_0^T \frac{1}{2}\left\|K_tu(\cdot,t)-f(\cdot,t)\right\|_2^2 + \alpha \mathcal{R}(u(\cdot,t))  + \beta \mathcal{S}(\boldsymbol{v}(\cdot,t)) dt\\
\text{s.t.} &\quad \mathcal{M}(u,\boldsymbol{v}) = 0\quad\text{ in } \mathcal{D'}(\left[0,T\right]\times\Omega).\notag
\end{align}

The first term in this functional acts as a data fidelity between the measured data $f$ and the objective function $u$ in the case of Gaussian noise. One may think of other data fidelities such as the $L^1$ distance for salt and pepper noise or the Kullback-Leibler divergence for Poisson noise.\\
The second term $\mathcal{R}$ in our general model constitutes a regularizer for the image sequence $u$. We mention that $R$ only acts on single time steps For reconstructing smooth images, the quadratic $L^2$-regularization on the gradient can be used, 
but a more natural regularization in the context of images is the total variation, which preserves edges to some extent and favors locally homogeneous intensities. The total variation coincides with the semi-norm on the space of functions with bounded variation and we set 
\begin{align*}
\mathcal{R}(u(\cdot,t)) = \left|u(\cdot,t)\right|_{BV}.
\end{align*} 
We mention that of course other higher-order versions of total variation (cf. \cite{benning2013higher, bredies2010total}) can be used for the regularization as well, with hardly any changes in the analysis due to the similar topological properties \cite{bredies2014regularization}. Regularizers for the velocity field $\boldsymbol{v}$ can be motivated very similar to those for images. An L$^2$ penalization of the gradient of $\boldsymbol{v}$
\begin{align*}
\mathcal{S}(\boldsymbol{v}(\cdot,t)) = \left\|\nabla \boldsymbol{v}(\cdot,t)\right\|_2^2
\end{align*}
leads to smooth velocity fields whereas a total variation based regularizer
\begin{align*}
\mathcal{S}(\boldsymbol{v}(\cdot,t)) = \left|\boldsymbol{v}(\cdot,t)\right|_{BV}
\end{align*}
favors piecewise constant flow fields. We mention that constraints such as an upper bound on the norm of $\boldsymbol{v}$ can be incorporated into $\mathcal{S}$ by adding the characteristic function of the constraint set.\\
The final ingredient is to connect image data and flow field by choosing a suitable constraint $\mathcal{M}(u,\boldsymbol{v})$. Using the brightness constancy assumption leads to the classical optical flow constraint and we set 
\begin{align*}
\mathcal{M}(u,\boldsymbol{v}) = u_t + \nabla u\cdot\boldsymbol{v}.
\end{align*}
More flexibility is given by the continuity equation
\begin{align*}
\mathcal{M}(u,\boldsymbol{v}) = u_t + \nabla\cdot(u\boldsymbol{v})
\end{align*}
that arises from the natural assumption that mass keeps constant over time. Both constraints add a non-linearity to the model, which leads to difficulties in the analysis arising from the  product $\nabla u\cdot\boldsymbol{v}$ resp. $\nabla\cdot(u\boldsymbol{v})$. Moreover, the model becomes non-convex and thus challenging from a numerical point of view, because local minimizers can appear. On the other hand the motion constraints and possibly strong regularization of motions provide a framework where motion estimation can enhance the image reconstruction and vice versa. In particular this makes the motion estimation more robust to noise (cp. Figure \ref{fig:plotAEEvsNoise}).
To simplify our notation, we use abbreviations of our model in the structure [\textit{regularizer u}]-[\textit{regularizer $\boldsymbol{v}$}] [\textit{constraint}] as e.g. for the TV-TV optical flow model.

\subsection{Preliminaries}
\label{notationsAndPreliminaries}
In what fallows we consider gray valued image data $u$ on a space-time domain $\Omega_T := \Omega\times\left[0,T\right]\subset\R{2}\times\R{+}, u : \Omega_T \rightarrow \R{}$. The sequence of flow-fields will be denoted by $\boldsymbol{v}$ and is defined on $\Omega_T$ with Neumann boundary conditions in space, $\boldsymbol{v} : \Omega_T \rightarrow \R{2}$. We expect finite speeds which gives the useful natural bound
\begin{align}
\left\|\boldsymbol{v}\right\|_\infty \leq c_\infty < \infty\quad \text{ a.e. in } \Omega\times\left[0,T\right].
\label{natuerlicheSchrankeV}
\end{align}
This assumption is reasonable since we have an application to real data (e.g. cell movement, car-tracking) in mind.

Besides this, a bound on the divergence of $\boldsymbol{v}$ in some Lebesgue space $\Theta$ is later needed in order to prove the existence of a minimizer. From the physical point of view the divergence measures the magnitude of the source or sink of $\boldsymbol{v}$. Consequently, having $\nabla\cdot\boldsymbol{v}\in L^p$ means an overall boundedness of sources and sinks, which is however not necessarily point wise Moreover, for the flow $\boldsymbol{v}$ the divergence is a measure for compressibility. We speak of an incompressible flow if $\nabla\cdot\boldsymbol{v} = 0$, so bounding the divergence means bounding the compressibility of $\boldsymbol{v}$.\\
For this paper required definitions can be found in the appendix. Moreover, we illustrate error measures for velocity fields and explain the discretization of our model in detail there. Finally, the appendix contains a pseudo-code and further results.

\section{Analytical results}
\label{section:analyticalResults}
The most challenging model from an analytical viewpoint is the joint TV-TV optical flow model
\begin{align}
\label{tvl2generalEnergy}
J(u,\boldsymbol{v}) = &\int_0^T \frac{1}{2}\left\|Ku-f\right\|_{\mathcal{F}}^2 + \alpha \left|u(\cdot,t)\right|_{BV}^p  + \beta \left| \boldsymbol{v}(\cdot,t)\right|_{BV}^q dt\\
&\text{s.t. } u_t + \nabla u\cdot\boldsymbol{v} = 0\quad\text{ in } \mathcal{D'}(\Omega_T),\notag\\
& \left\|\boldsymbol{v} \right\|_\infty \leq c_\infty,\left\|\nabla\cdot\boldsymbol{v}\right\|_{\Theta} \leq c^*\notag
\end{align} 
for $p>1$ and $q\geq 1$. For simplicity we restrict to the TV-TV optical flow model spatial dimension two here. We refer to \cite{dirks} for the full analysis including the mass preservation constraint and L$^2$ regularization. We want to mention that our results apply for any regularizers $\mathcal{R},\mathcal{S}$ satisfying $\mathcal{R}(u(\cdot,t)) \geq \left|u(\cdot,t)\right|_{BV}^p$ and $\mathcal{S}(\boldsymbol{v}(\cdot,t)) \geq \left|\boldsymbol{v}(\cdot,t)\right|_{BV}^q$. For the bounded linear operator $K$ we assume $K : L^1(\Omega)\rightarrow\mathcal{F}(\Omega)$ with some Hilbert space $\mathcal{F}(\Omega)$. Moreover, we want to mention that the operator $K$ operates on single time steps only, however the analysis can be generalized for time-dependent $K$ (cp. comment). Note that due to the embedding of Sobolev spaces $W^{1,s}$ into $BV$, the results can also be generalized to other gradient regularizations with $s > 1$. Finally, we mention that in the case of the continuity equation as a constraint the results can even be obtained under weaker conditions if the continuity equation is considered in a weak form, we refer to \cite{dirks} for further details. 

Note that for the following analysis the bound on the divergence of $\boldsymbol{v}$ is crucial. The chosen bound for $\boldsymbol{v}$ induces a condition on the space $\Theta$ for which we will need to assume that
\begin{equation} \label{eq:divergenceassumption}
\exists s>1,k>2:\quad L^{p^*s}(0,T;L^{k}(\Omega))^2 \hookrightarrow  \Theta,
\end{equation}
with $p^*$ being the H{\"o}lder conjugate of $p$.

Our main result in this section is the following:
\begin{theorem}{} \label{theorem:existenceMinimizerTVTVmodel}
	Let $\Omega\subset\R{2},p,q > 1, \hat{p} = \min\left\{p,2\right\}$, $K\boldsymbol{1}\neq 0$ and \eqref{eq:divergenceassumption} be satisfied.
	Then there exists a minimizer of \eqref{tvl2generalEnergy} in the space
	\begin{align*}
	\left\{ (u,\boldsymbol{v}) : u\in L^{\hat{p}}(0,T;BV(\Omega)) , \boldsymbol{v}\in L^{q}(0,T;BV(\Omega))^2,\nabla\cdot\boldsymbol{v}\in \Theta \right\}.
	\end{align*}
\end{theorem}
The proof theorem \ref{theorem:existenceMinimizerTVTVmodel} is based on an application of weak lower semi-continuity and compactness techniques. It follows from the following three properties verified in the next sections:
\begin{enumerate}
	\item Weak-star compactness of sublevel sets (coercivity),
	\item Weak-star lower semi continuity,
	\item Closedness of the constraint set via convergence in a distributional sense.
\end{enumerate}
\vspace{0.5cm}
\textbf{Comment:} For a time-dependent linear operator $K$, most arguments can be used in an analogue fashion. The proof even simplifies if the stronger regularity assumption $\|Ku-f\|_{L^2\times[0,T]}$ holds, since we do not have to start our argumentation from single time steps following with boundedness for their time integral.

\subsection{Coercivity and lower-semi continuity}
We mention that coercivity and lower semi continuity are independent of the constraint.

\begin{lemma}{Coercivity}\\
	Let $\hat{p} = \min\left\{p,2\right\}$,  $K\boldsymbol{1}\neq 0$, and $(u,\boldsymbol{v})$ be such that
	\begin{align*}
J(u,\boldsymbol{v})\leq\nu, \qquad 	\left\|\boldsymbol{v}\right\|_\infty \leq c_\infty 
	\end{align*}
	Then there exists $c\in\R{}$ such that
	\begin{align*}
	\left\|u\right\|_{L^{\hat{p}}(0,T;BV(\Omega))}\leq c,\quad \left\|\boldsymbol{v}\right\|_{L^q(0,T;BV(\Omega))}\leq c,
	\end{align*}
	and consequently, the sublevel set $\mathcal{S}_{\nu}$ (see \ref{def:sublevelset}) is not empty and compact in the weak-star topology of $L^{\hat{p}}(0,T;BV(\Omega))\times L^q(0,T;BV(\Omega))^2$.
	\label{lemma:coercivitytvtvgeneralmodel}
\end{lemma}
\begin{proof}
	 We begin with the bound for $u$ and have to prove that for arbitrary $u\in L^p(0,T;BV(\Omega))$ with $J(u,\cdot)\leq\nu$ we have 
	\begin{align}
	\left\|u\right\|_{L^{\hat{p}}(0,T;BV(\Omega))}^{\hat{p}} \leq 2^{\hat{p}-1} \left(\int_0^T \left\|u\right\|_{L^1(\Omega)}^{\hat{p}}\dif t + \int_0^T\left|u\right|_{BV(\Omega)}^{\hat{p}}\dif t\right)\leq c\label{sumCoercivityU}.
	\end{align}
	To deduce this bound we need to estimate each of the two terms in the last line of the inequality. \\
	Since all three terms in energy \eqref{tvl2generalEnergy} are positive, from $J(u,\boldsymbol{v})\leq\nu$ we directly get a bound on each of the three parts. It follows that 
	\begin{align*}
	\left\|Ku-f\right\|_{L^2(0,T;L^2(\Omega))}\leq\nu,
	\end{align*}
	which naturally implies 
	\begin{align*}
	\left(Ku(\cdot,t)-f(\cdot,t)\right)\in L^2(\Omega) \quad \text{ a.e. in }\left[0,T\right].
	\end{align*}
	Consequently, $\left\|Ku(\cdot,t)-f(\cdot,t)\right\|_{L^2(\Omega)}$ is bounded almost everywhere in $t\in\left[0,T\right]$ and we define
	\begin{align*}
	c_K(t) := \left\|Ku(\cdot,t)-f(\cdot,t)\right\|_{L^2(\Omega)}.
	\end{align*}
	We want to emphasize here that $c_K(t)$ gives a constant for every time step $t\in\left[0,T\right]$, but the integral $\int_0^T c_K^{\hat{p}}\dif t$ is only bounded for $1< {\hat{p}}\leq 2$ due to the $L^2$-regularity in time.\\
	Proceeding now to Equation \eqref{sumCoercivityU} we directly get from $J(u,\cdot)\leq\nu$ that
	\begin{align*}
	\int_0^T\left|u\right|_{BV(\Omega)}^p\dif t =\int_0^T TV(u)^p\dif t\leq\nu.
	\end{align*}
	Consequently, the crucial point is to find a bound for $\left\|u\right\|_{L^p(0,T;L^1(\Omega))}$. Let $t\in\left[0,T\right]$ be an arbitrary time step. First, we deduce a bound for this single time step $\left\|u(\cdot,t)\right\|_{L^1(\Omega)}$ and start with a decomposition for $u$: 
	\begin{align*}
	\bar{u}=\frac{1}{\left|\Omega\right|}\int_\Omega u(x,t)\dif x,\quad u_0=u(\cdot,t)-\bar{u}.
	\end{align*}
	From this definition it follows directly that $u_0$ fulfills 
	\begin{align*}
	\int_\Omega u_0\dif x=0\quad\text{(mean value zero)},
	\end{align*}
	and $TV(u(\cdot,t))=TV(u_0)\leq\nu$. Using the Poincar\'{e}-Wirtinger inequality \cite{meyers1977integral} we obtain an $L^2$-bound for $u_0$:
	\begin{align*}
	\left\| u_0\right\|_{L^2(\Omega)}\leq c_3 TV(u_0)\leq c_3 \nu,
	\end{align*}
	where $c_1,c_2$ and $c_3$ are positive constants. Moreover, we need a bound for $\left\|K\bar{u}\right\|_{L^2(\Omega)}$, which we get by calculating
	\begin{align*}
	\left\|K\bar{u}\right\|_{L^2}^2 - 2\left\|K\bar{u}\right\|_{L^2}\left(\left\|K\right\|\left\|u_0\right\|_{L^2}+\left\|f\right\|_{L^2}\right)
	&\leq \left\|K\bar{u}\right\|_{L^2}^2-2\left\|Ku_0-f\right\|_{L^2}\left\|K\bar{u}\right\|_{L^2} \\
	&\leq \left\|Ku_0-f\right\|_{L^2}^2+\left\|K\bar{u}\right\|_{L^2}^2-2\left\|Ku_0-f\right\|_{L^2}\left\|K\bar{u}\right\|_{L^2}\\
	&= \left( \left\|Ku_0-f\right\|_{L^2}-\left\|K\bar{u}\right\|_{L^2}\right)^2 \leq \left\|Ku_0 + K\bar{u}-f\right\|_{L^2}^2 \\
	&= \left\|Ku(\cdot,t)-f(\cdot,t)\right\|_{L^2}^2\leq c_K(t)^2.
	\end{align*}
	Defining $x:=\left\|K\bar{u}\right\|_{L^2(\Omega)},a:=\left\|K\right\|\left\|u_0\right\|_{L^2(\Omega)}+\left\|f\right\|_{L^2(\Omega)}$, we get the simple quadratic inequality
	\begin{align}
	\label{quadratischeGleichung1}
	x^2-2xa\leq c_K(t)^2
	\end{align}
	and furthermore know
	\begin{align*}
	0\leq a \leq \left\|K\right\|c_3 \nu+\left\|f\right\|_{L^2(\Omega)}=:c_4.
	\end{align*}
	Plugging this into the quadratic inequality (\ref{quadratischeGleichung1}) yields the solution
	\begin{align*}
	0\leq x \leq c_4+\sqrt{\nu+c_4^2} \leq c_4 + c_7 (c_K(t)+c_4).
	\end{align*}
	The assumption $K\boldsymbol{1}\neq 0$ leads to an estimate for the operator
	\begin{align*}
	&\left\|K\bar{u}\right\|_{L^2(\Omega)} = \left|\frac{1}{\left|\Omega\right|}\int_\Omega u\dif x\right|\left\|K\boldsymbol{1}\right\|_{L^2(\Omega)}\leq c_4+\sqrt{\nu+c_4^2} \\
	&\Leftrightarrow \left|\frac{1}{\left|\Omega\right|} \int_\Omega u\dif x\right| \leq \frac{c_4 + c_7 (c_K(t)+c_4)}{\left\|K\boldsymbol{1}\right\|_{L^2(\Omega)}} =: c_5(t).
	\end{align*}
	We are now able to bound the $L^1$-norm of a single time step $t\in\left[0,T\right]$ by a constant $c_u(t)$ as follows:
	\begin{align*}
	0&\leq\left\|u(\cdot,t)\right\|_{L^1(\Omega)} \leq c_6\left\|u(\cdot,t)\right\|_{L^2(\Omega)} = c_6\left\|u_0+\frac{1}{\left|\Omega\right|}\int_\Omega u(x,t)\dif x  \right\|_{L^2(\Omega)} \\
	&\leq c_6 \left( \left\|u_0\right\|_{L^2(\Omega)}  + \left|\frac{1}{\left|\Omega\right|} \int_\Omega u(x,t)\dif x\right|\right) 
	\leq c_6 \left( c_3\nu + c_5(t)\right)=: c_u(t).
	\end{align*}
	Since we are integrating over all these constants $c_u(t)$ and the integral is only bounded for $1< \hat{p} \leq 2$, we see that the assumption on $\hat{p}$ and $p$ is crucial. Consequently, we have
	\begin{align*}
	\int_0^T \left\|u(\cdot,t)\right\|_{L^1(\Omega)}^{\hat{p}}\dif t \leq \int_0^T c_u(t)^{\hat{p}}\dif t \leq c_M.
	\end{align*}
	Combining both estimates we conclude with the required bound for arbitrary $u\in L^{\hat{p}}(0,T;BV(\Omega))$:
	\begin{align*}
	\left\|u\right\|_{L^{\hat{p}}(0,T;BV(\Omega))}^{\hat{p}} = \int_0^T \left\|u\right\|_{BV(\Omega)}^{\hat{p}} \dif t 
	\leq \int_0^T \left\|u\right\|_{L^1(\Omega)}^{\hat{p}}\dif t + \int_0^T\left|u\right|_{BV(\Omega)}^{\hat{p}}\dif t\leq c_MT +\nu.
	\end{align*}
	A bound for $\boldsymbol{v}$ is easier to establish, since we have $\left\|\boldsymbol{v}\right\|_{L^\infty(\Omega)}\leq c_\infty$ (see Equation \ref{natuerlicheSchrankeV}) almost everywhere. Similar to $u$, from $J(u,\boldsymbol{v})\leq\nu$  we obtain the a-priori bound
	\begin{align*}
	\int_0^T\left|\boldsymbol{v}(\cdot,t)\right|_{BV(\Omega)}^{q}\dif t\leq\nu
	\end{align*}
	for $\boldsymbol{v}$ from Equation ($\ref{tvl2generalEnergy}$). We calculate the bound for $\boldsymbol{v}$ directly as
	\begin{align*}
	\left\|\boldsymbol{v}\right\|_{L^q(0,T;BV(\Omega))}^q &= \int_0^T \left\|\boldsymbol{v}(\cdot,t)\right\|_{BV(\Omega)}^q \dif t \leq \int_0^T \left\|\boldsymbol{v}(\cdot,t)\right\|_{L^1(\Omega)}^q\dif t + \int_0^T\left|\boldsymbol{v}(\cdot,t)\right|_{BV(\Omega)}^{q}\dif t\\
	&\leq \int_0^T c_{\boldsymbol{v}}^q\left|\Omega\right|\dif t +\nu = c_{\boldsymbol{v}}^q\left|\Omega\right|T+\nu,
	\end{align*}
	where we have used the L$^\infty$-bound on $\boldsymbol{v}$. Combining the bounds for $u$ and $\boldsymbol{v}$,  we conclude with an application of the Banach-Alaoglu Theorem (see for example \cite{rudin1973functional}), which yields the required compactness result in the weak-star topology:\\
	It can be shown that $BV(\Omega)$ is the dual space of a Banach space $\CY$ (see \cite{burgerlevel2008}). From duality theory of Bochner spaces (cf. \cite{cengiz1998dual}) we get
	\begin{align*}
	L^{\hat{p}}(0,T;BV(\Omega)) = L^{p^*}(0,T;\CY(\Omega))^*,
	\end{align*}
	where $p^*$ is the H{\"o}lder-conjugate of p. With the same argumentation we get
	\begin{align*}
	L^q(0,T;BV(\Omega)) = L^{q^*}(0,T;\CY(\Omega))^*.
	\end{align*}
	Since both spaces are duals, an application of the Banach-Alaoglu Theorem yields the compactness in the weak-star topology.
\end{proof}

\begin{lemma}
	\label{lsctvl2generalmodel}
The functional $J$ is lower semi continuous with respect to the weak-star topology of 	$L^{\hat{p}}(0,T;BV(\Omega)) \times L^q(0,T;BV(\Omega))^2$.
\end{lemma}
\begin{proof}
	Norms and affine norms as well as their powers with exponent larger equal to one are always convex. Convex functionals on Banach spaces can be proven to be weakly lower semi continuous Due to the reflexivity of $L^2$ we directly obtain weak-star lower semi continuity\\
	Furthermore, it can be shown \cite{burgerlevel2008} that the total variation is weak-star lower semi continuous This property holds for exponentials $p$ of TV satisfying $p> 1$. \\
	Lower semi continuity holds for sums of lower semi continuous functionals, which concludes the proof.
\end{proof}

\subsection{Convergence of the constraint}
For completing the existence prove we have to deduce closedness of the constraint set. Consider admissible sequences $u^n$ and $\boldsymbol{v}^n$ in a sub-level set of $J$. From the regularization we obtain boundedness and consequently weak$^*$ convergence 
\begin{align*}
u^n\rightharpoonup^* u, \quad\boldsymbol{v}^n\rightharpoonup^* \boldsymbol{v}.
\end{align*}
In this context, the most challenging point is to prove convergence (in at least a distributional sense) of the constraint 
\begin{align*}
(u_t)^n + \nabla u^n\cdot\boldsymbol{v}^n\rightarrow u_t + \nabla u\cdot\boldsymbol{v}.
\end{align*}
The major problem arises from the product $\nabla u^n\cdot\boldsymbol{v}^n$, which does not necessarily converge to the product of their individual limits. A counterexample can be found in \cite{tartar1983compensated}. To achieve convergence we need at least one of the factors to converge strongly, but this cannot be deduced from boundedness directly. A way out gives the Aubin-Lions Theorem \cite{AubinLion1,AubinLion2,AubinLion3} which yields a compact embedding
\begin{align*}
L^p(0,T;\CX)\subset\subset L^p(0,T;\CY),
\end{align*}
and hence strong convergence in $\CY$, if $u^n$ is bounded in $L^p(0,T;\CX)$ and $(u_t)^n$ is bounded in $L^r(0,T;\CZ)$ for some $r$ for Banach spaces $\CX\subset\subset\CY\hookrightarrow\CZ$. Applied to our case we set $\CX=BV(\Omega)$ and $\CY=L^r(\Omega)$. The first goal is to calculate a bound for $(u_t)^n$ in some Lebesgue space $L^r(0,T;\CZ)$, which is given by the following lemma:
\begin{lemma}{Bound for $u_t$}\\
	\label{lemma:bounddtu}
	Let $\Omega\subset\mathbb{R}^2,u\in L^p(0,T;BV(\Omega)), \boldsymbol{v}\in L^{q}(0,T;BV(\Omega))$ such that $u_t + \nabla u\cdot\boldsymbol{v} = 0$ and 
	\begin{align*}
	\left\|\boldsymbol{v}\right\|_\infty \leq c_\infty
	\end{align*}
	Let furthermore $\nabla\cdot\boldsymbol{v}\in L^{p^*s}(0,T;L^{2k}(\Omega))$ with $k>1,s>1$ and let $p^*$ denote the H{\"o}lder conjugate of $p$.
	Then we have
	\begin{align*}
	u_t \in L^{\frac{ps}{p+s-1}}(0,T;L^{\frac{2k}{k+1}}(\Omega))
	\end{align*}
	with uniform bounds.
\end{lemma}
\begin{proof}
	Our goal is to show that a sequence $u^n$ (we will omit the lower $n$ in the following), satisfying the optical flow equation, acts as a bounded linear functional in some Bochner-space, thus being an element of the corresponding dual space. 
	We write down the weak form of the optical flow equation with some test function $\varphi$
	\begin{align*}
	\left|\int_0^T\int_\Omega u_t\varphi\dif x \dif t\right| &= \left| \int_0^T\int_\Omega u\nabla\cdot(\boldsymbol{v}\varphi) \dif x \dif t\right|
	\underbrace{\leq}_{\text{H\"older}}  \int_0^T \left(\int_\Omega u^2 \dif x \right)^\frac{1}{2} \left( \int_\Omega (\nabla\cdot(\boldsymbol{v}\varphi))^2\dif x \right)^\frac{1}{2} \dif t  \\
	&\underbrace{\leq}_{\text{Mink.}} \int_0^T \left\|u\right\|_{L^2} \left[\left( \int_\Omega (\varphi\nabla\cdot \boldsymbol{v})^2\dif x \right)^\frac{1}{2} + \left( \int_\Omega (\boldsymbol{v}\cdot\nabla\varphi)^2\dif x \right)^\frac{1}{2} \right] \dif t\\
	&= \underbrace{\int_0^T \left\|u\right\|_{L^2} \left( \int_\Omega (\varphi\nabla\cdot \boldsymbol{v})^2\dif x \right)^\frac{1}{2} \dif t}_{(i)} + \underbrace{\int_0^T \left\|u\right\|_{L^2} \left( \int_\Omega (\boldsymbol{v}\cdot\nabla\varphi)^2\dif x \right)^\frac{1}{2}  \dif t}_{(ii)}
	\end{align*}
	Let us start with an estimate for part $(i)$, which we obtain after three subsequent applications of the H{\"o}lder inequality:
	\begin{align*}
	\int_0^T \left\|u\right\|_{L^2} \left( \int_\Omega (\varphi\nabla\cdot \boldsymbol{v})^2\dif x \right)^\frac{1}{2} \dif t \leq  \left\|u\right\|_{L^p(0,T;L^2)} \left\|\nabla\cdot\boldsymbol{v}\right\|_{L^{p^*s}(0,T;L^{2k})} \left\|\varphi\right\|_{L^{p^*s^*}(0,T;L^{2k^*})},
	\end{align*}
	for $p^*$ and $s$ as in the statement of the theorem. An estimate for part $(ii)$ follows with Cauchy-Schwarz and H{\"o}lder:
	\begin{align}
		\int_0^T \left\|u\right\|_{L^2} \left( \int_\Omega (\boldsymbol{v}\cdot\nabla\varphi)^2\dif x \right)^\frac{1}{2}  \dif t&\leq c_{\boldsymbol{v}}\int_0^T \left\|u\right\|_{L^2}  \left\|\varphi\right\|_{W^{1,2}} \dif t\notag
		\leq c_{\boldsymbol{v}} \left\|u\right\|_{L^p(0,T;L^2)} \left\|\varphi\right\|_{L^{p^*}(0,T;W^{1,2})}\notag
	\end{align}
	Combining these estimates we obtain
	\begin{align*}
	\left|\int_0^T\int_\Omega u_t\varphi\dif x \dif t\right| &\leq \left\|u\right\|_{L^p(0,T;L^2)} \left\|\nabla\cdot\boldsymbol{v}\right\|_{L^{p^*s}(0,T;L^{2k})} \left\|\varphi\right\|_{L^{p^*s^*}(0,T;L^{2k^*})} \\
	&+c_{\boldsymbol{v}} \left\|u\right\|_{L^p(0,T;L^2)} \left\|\varphi\right\|_{L^{p^*}(0,T;W^{1,2})}\\
	&\leq (\left\|\nabla\cdot\boldsymbol{v}\right\|_{L^{p^*s}(0,T;L^{2k})} + c_{\boldsymbol{v}}) \left\|u\right\|_{L^p(0,T;L^2)} \left\|\varphi\right\|_{L^{p^*s^*}(0,T;L^{2k^*})}.
	\end{align*}
	In the first inequality, we used the embedding 
	\begin{align*}
	W^{1,2}(\Omega)\hookrightarrow L^{2k^*}(\Omega),\quad\forall 1\leq k^*<\infty.
	\end{align*}
	The sum of the first terms is bounded because of the assumptions made above. A bound for $u$ follows again from the $BV(\Omega)$ embedding into $L^2(\Omega)$ and we conclude  
	\begin{align*}
	\left\langle u_t,\varphi \right\rangle := \int_0^T\int_\Omega u_t\varphi\dif x \dif t \leq  C \left\|\varphi\right\|_{L^{p^*s^*}(0,T;L^{2k^*})}.
	\end{align*}
	Thus, $u_t$ forms a bounded linear functional on $L^{p^*s^*}(0,T;L^{2k^*}(\Omega))$ and we end up with
	\begin{align*}
	u_t\in (L^{p^*s^*}(0,T;L^{2k^*}(\Omega)))^* = L^{\frac{ps}{p+s-1}}(0,T;L^{\frac{2k}{k+1}}(\Omega)).
	\end{align*}
\end{proof}

\vspace{1cm}

\begin{theorem}{Compact embedding for $u$}\\
	\label{theorem:compactEmbeddingFuerU}
	Let the assumptions of Lemma \ref{lemma:bounddtu} be fulfilled. Then the set
	\begin{align*}
	\left\{u \in L^p(0,T;BV(\Omega)): \Vert u \Vert \leq C, u_t+\nabla u\cdot\boldsymbol{v}=0 \right\}
	\end{align*}
	can be compactly embedded into
	$L^p(0,T;L^r(\Omega))$
	for $\frac{2k}{k+1}\leq r<2$, and $k$ as given in constraint $\eqref{eq:divergenceassumption}$.
	
\end{theorem}
\begin{proof}
	We have a natural a-priori estimate for $u$ in $L^p(0,T;BV(\Omega))$. We moreover deduced a bound for $u_t$ in $L^{\frac{ps}{p+s-1}}(0,T;L^{\frac{2k}{k+1}}(\Omega))$. Embeddings of $BV(\Omega)$ into $\CY=L^r(\Omega)$ are compact for $r<\frac{n}{n-1}$, where $n$ is the spatial dimension. Moreover, the embedding $L^r(\Omega)\hookrightarrow L^{\frac{2k}{k+1}}(\Omega)$ is continuous for $\frac{2k}{k+1} \leq r$. Combining this we see that embedding
	\begin{align*}
	BV(\Omega)\subset\subset L^r(\Omega)\hookrightarrow L^{\frac{2k}{k+1}}(\Omega),
	\end{align*}
	is fulfilled for all $\frac{2k}{k+1}\leq r<2$. An application of the Aubin-Lions lemma \ref{AubinLionLemma} yields the compact embedding
	\begin{align*}
	\left\{u : u \in L^p(0,T;BV(\Omega)), u_t+\nabla u\cdot\boldsymbol{v}=0 \right\}\subset\subset L^p(0,T;L^r(\Omega)).
	\end{align*}
\end{proof}

Another application of this fairly general result can be found in \cite{dirks}. With this compact embedding result we conclude with strong convergence for $u^n \rightarrow u$ in $ L^p(0,T;L^r(\Omega))$ and are now able to prove convergence of the product $\nabla u^n\cdot\boldsymbol{v}^n$, to the product of their individual limits $\nabla u\cdot\boldsymbol{v}$.
\vspace{5mm}
\begin{lemma}{Convergence of the constraint}\\
	\label{convergenceConstraint}
	Let $\Omega\subset\mathbb{R}^2,u^n\in L^p(0,T;BV(\Omega))$ and $\boldsymbol{v}^n\in L^{q}(0,T;BV(\Omega))^2$ be bounded sequences. Let furthermore the assumptions of Lemma \ref{lemma:bounddtu} be fulfilled. Then
	\begin{align*}
	(u_t)^n+\nabla u^n\cdot \boldsymbol{v^n} \rightharpoonup u_t+\nabla u\cdot \boldsymbol{v}
	\end{align*}
	in the sense of distributions.
\end{lemma}
\begin{proof}
	
	For the following proof let $\varphi\in C_0^\infty(\Omega), u^n\in L^p(0,T;BV(\Omega))$ and $\boldsymbol{v}^n\in L^q(0,T;BV(\Omega))^2$. For the time derivative we simply calculate
	\begin{align*}
	\int_0^T\int_\Omega\left((u_t)^n-u_t\right)\varphi\dif x\dif t = - \int_0^T\int_\Omega\left(u^n-u\right)\varphi_t\dif x\dif t \rightarrow 0.
	\end{align*}
	Since test functions are dense in the dual space of $u$, we directly obtain convergence from the weak convergence $u^n\rightharpoonup u$. For the second part we begin with an analogous argument and estimate 
	\begin{align*}
	-\int_0^T\int_\Omega\left(\nabla u^n\cdot \boldsymbol{v}^n - \nabla u\cdot \boldsymbol{v}\right)\varphi \dif x \dif t &= \int_0^T\int_\Omega u^n\nabla\cdot(\varphi \boldsymbol{v}^n) - u\nabla\cdot(\varphi \boldsymbol{v}) \dif x \dif t\\
	=\underbrace{\int_0^T\int_\Omega (u^n-u)\nabla\cdot(\varphi \boldsymbol{v}^n)\dif x \dif t}_{(i)} &+ \underbrace{\int_0^T\int_\Omega u\nabla\cdot(\varphi (\boldsymbol{v}^n-\boldsymbol{v})) \dif x \dif t}_{(ii)}.
	\end{align*}
	Part $(i)$ can be estimated as follows:
	\begin{align*}
	\int_0^T\int_\Omega (u^n-u)\nabla\cdot(\varphi \boldsymbol{v}^n)\dif x \dif t &\leq \left\| u^n-u \right\|_{L^p(0,T;L^r)} \left\|\varphi\nabla\cdot \boldsymbol{v}^n + \boldsymbol{v}^n \cdot\nabla\varphi \right\|_{L^{p^*}(0,T;L^{r^*})}\\
	\leq  \left\| u^n-u \right\|_{L^p(0,T;L^r(\Omega))}\cdot 
	&(\left\|\varphi\nabla\cdot \boldsymbol{v}^n \right\|_{L^{p^*}(0,T;L^{r^*}(\Omega))} + \left\| \boldsymbol{v}^n \cdot\nabla\varphi \right\|_{L^{p^*}(0,T;L^{r^*}(\Omega))})\\
	\leq  \left\| u^n-u \right\|_{L^p(0,T;L^r(\Omega))}\cdot
	&(\underbrace{ \|\varphi\|_{L^{p^*s^*}(0,T;L^{r^*}(\Omega))} \|\nabla\cdot \boldsymbol{v}^n \|_{L^{p^*s}(0,T;L^{r^*}(\Omega))}}_{(i.1)} + \underbrace{\left\| \boldsymbol{v}^n \cdot\nabla\varphi \right\|_{L^{p^*}(0,T;L^{r^*}(\Omega))}}_{(i.2)})
	\end{align*}
	$(i.1):$ $\varphi$ is a test function and therefore bounded. From the assumptions we also have $\nabla\cdot\boldsymbol{v}^n\in L^{p^*s}(0,T;L^{2k}(\Omega))^2$. Consequently, we have to prove
	\begin{align*}
	L^{p^*s}(0,T;L^{2k}(\Omega)) \hookrightarrow L^{p^*s^*}(0,T;L^{r^*}(\Omega)).
	\end{align*}
	In terms of the embedding theory of Lebesgue spaces we show that $2k\geq r^*$. At this point it is important to keep in mind that $r$ and $r^*$ are H{\"o}lder-conjugated and the embedding theorem for optical flow allows $\frac{2k}{k+1}\leq r < 2$. The condition $2k\geq r^*$, on the other hand, translates to $\frac{2k}{2k-1}\leq r$, which is smaller than 2 for all $k>1$. By taking $\max(\frac{2k}{k+1},\frac{2k}{2k-1})\leq r$ both on $r$ conditions are satisfied. This yields the required bound for $\nabla\cdot\boldsymbol{v}^n$ in $L^{p^*s^*}(0,T;L^{r^*}(\Omega))$.\\
	$(i.2):$ This part is bounded by a constant due to the boundedness of $\boldsymbol{v}^n$ and the characteristics of $\varphi$. Following the arguments for strong convergence of $u$ from above, we conceive that $(i)$ tends to zero.\\
	\vspace{5mm}
	Estimating part $(ii)$ again requires Lebesgue embedding theory, since
	\begin{align*}
	\int_0^T\int_\Omega u\nabla\cdot(\varphi (v^n-v)) \dif x \dif t &= \int_0^T\int_\Omega \underbrace{u \varphi \nabla\cdot (v^n-v)}_{(ii.1)} + \underbrace{u (v^n-v)\cdot\nabla\varphi}_{(ii.2)}\dif x \dif t.
	\end{align*}
	$(ii.1):$ Using Lebesgue embedding theory we show
	\begin{align*}
	L^p(0,T;BV(\Omega)) &\hookrightarrow L^p(0,T;L^{\frac{2k}{2k-1}}(\Omega)) = L^p(0,T;L^{(2k)^*}(\Omega)) \\
	&\hookrightarrow L^{\frac{ps}{ps-p+1}}(0,T;L^{(2k)^*}(\Omega)) = L^{(p^*s)^*}(0,T;L^{(2k)^*}(\Omega)).
	\end{align*}
	Consequently, $u\in L^{(p^*s)^*}(0,T;L^{(2k)^*}(\Omega))$, which is the dual of $L^{p^*s}(0,T;L^{2k}(\Omega))$. Due to the weak-star convergence of $\nabla\cdot\boldsymbol{v}^n$ part $(ii).1$ tends to $0$ as $n\rightarrow\infty$.\\
	$(ii.2):$ The boundedness of $\boldsymbol{v}$ gives us $\boldsymbol{v}\in L^{\infty}(\left[0,T\right]\times\Omega)$ and a-priori weak-star convergence. Consequently, we need $u\nabla\varphi \in L^1(\left[0,T\right]\times\Omega)$. Due to the compact embedding $BV(\Omega)\subset\subset L^1(\Omega)$ and $p>1$ we get $L^p(0,T;BV(\Omega)) \hookrightarrow L^1(0,T;L^1(\Omega))$.
	This gives us $u\in L^1(\left[0,T\right]\times\Omega)$ and since test functions are dense in $L^1$ we end up with the required $u\nabla\varphi \in L^1(\left[0,T\right]\times\Omega)$.

	Putting all arguments together we end up with convergence of the constraint
	\begin{align*}
	&\lim_{k\rightarrow\infty} \left|\int_0^T\int_\Omega\left(\nabla u^n\cdot \boldsymbol{v}^n - \nabla u\cdot \boldsymbol{v}\right)\varphi \dif x \dif t\right| 
	\leq  \ C \lim_{k\rightarrow\infty}\left\| u^n-u \right\|_{L^p(0,T;L^r)}\\ +& \lim_{k\rightarrow\infty}\left|\int_0^T\int_\Omega u \varphi \nabla\cdot (\boldsymbol{v}^n-\boldsymbol{v})  \dif x \dif t\right|
	 + \lim_{k\rightarrow\infty}\left|\int_0^T\int_\Omega u (\boldsymbol{v}^n-\boldsymbol{v})\cdot\nabla\varphi \dif x \dif t\right|=  0.
	\end{align*}
	
\end{proof}

\section{Primal-dual numerical realization}
\label{section:numericalFramework}
Similar to the analytical part, we illustrate the numerical realization of the joint TV-TV optical flow model. Numerical schemes for the other models can be derived with only minor changes. We refer to \cite{dirks} for details. The proposed energy \eqref{tvl2generalEnergy} is a constrained minimization problem. The constraints on $\boldsymbol{v}$ and $\nabla\cdot\boldsymbol{v}$ are technical assumptions for the analysis of the model, where the bounds can be chosen arbitrarily large. Therefore, we neglect them in the following in the numerical considerations. Introducing an $L^1$ penalty term for the optical flow constraint with additional weight $\gamma$ leads to the unconstrained joint minimization problem
\begin{align}
\argmin_{u,\boldsymbol{v}} &\int_0^T \frac{1}{2}\left\|Ku-f\right\|_2^2 + \alpha \left\|\nabla u\right\|_1 + \beta\left\|\nabla \boldsymbol{v}\right\|_1 + \gamma \left\|u_t+\nabla u\cdot\boldsymbol{v}\right\|_1 \dif t.
\label{tvtvOpticalFlowUnconstrained}
\end{align}
By taking into account the bounds to $\boldsymbol{v}$ and $\nabla\cdot\boldsymbol{v}$ the minimizer of \eqref{tvtvOpticalFlowUnconstrained} converges for $\gamma\rightarrow\infty$ to the minimizer of Theorem \ref{theorem:existenceMinimizerTVTVmodel}. \\
Due to the dependence of energy \eqref{tvtvOpticalFlowUnconstrained} on the product of $\nabla u\cdot \boldsymbol{v}$ the energy is non-linear and therefore non-convex. Moreover the involved L$^1$ norms are non-differentiable and we have linear operators acting on $u$ and $\boldsymbol{v}$. Hence, minimizing the energy is numerically challenging.\\ 
We propose an alternating minimization technique, switching between minimizing with respect to $u$ and with respect to $\boldsymbol{v}$, while fixing the other variable. This leads to the following two-step scheme
\begin{align}
u^{k+1} = \argmin_{u} &\int_0^T \frac{1}{2}\left\|Ku-f\right\|_2^2  + \alpha \left\|\nabla u\right\|_1 + \gamma \left\|u_t+\nabla u \cdot \boldsymbol{v}^k\right\|_1 \dif t\label{eq:simultanTVTVunconstrainedSubU},\\
\boldsymbol{v}^{k+1} = \argmin_{\boldsymbol{v}} &\int_0^T  \left\|u^{k+1}_t+\nabla u^{k+1}\cdot \boldsymbol{v}\right\|_1 + \frac{\beta}{\gamma}\left\|\nabla \boldsymbol{v}\right\|_1 \dif t, \label{eq:simultanTVTVunconstrainedSubV}
\end{align}
where each of the subproblems is now convex and a primal-dual algorithm \cite{chambolle2011first,pock2009algorithm} can be applied. 

\paragraph{Problem in $u$:}
Illustrating the problem in $u$, we have to solve a classical ROF-problem \cite{rudin1992nonlinear} coupled with an additional transport term arising from the optical flow component. Each of the terms contains an operator and is therefore dualized. We set
\begin{align*}
F(Cu) := \int_0^T \frac{1}{2}\left\|Ku-f\right\|_2^2 + \alpha \left\|\nabla u\right\|_1 + \gamma\left\|(D_t + v_1^kD_x + v_2^kD_y)u\right\|_1 \dif t,
\end{align*}
with an underlying operator 
\begin{align*}
Cu=\begin{pmatrix}K\\\nabla\\D_t + v_1^kD_x + v_2^kD_y\end{pmatrix}u
\end{align*}
We first write down the convex conjugate $F^*$ corresponding to $F$:
\begin{align*}
F^*(\boldsymbol{y}) = \int_0^T \frac{1}{2}\left\|y_1\right\|_2^2 + \langle y_1,f\rangle +\alpha \delta_{B(L^\infty)}(y_2/\alpha) +\gamma \delta_{B(L^\infty)}(y_3/\gamma) \dif t.
\end{align*}
This leads to the primal-dual problem:
\begin{align*}
\argmin_{u}\argmax_{\boldsymbol{y}} \int_0^T\langle Cu,\boldsymbol{y}\rangle -\frac{1}{2}\left\|y_1\right\|_2^2 - \langle y_1,f\rangle -\alpha \delta_{B(L^\infty)}(y_2/\alpha) -\gamma \delta_{B(L^\infty)}(y_3/\gamma) \dif t.
\end{align*}
Plugging this into the primal-dual algorithm yields the following iterative systems consisting mostly of proximity problems
\begin{align*}
\tilde{\boldsymbol{y}}^{l+1} &= \boldsymbol{y}^{l} + \sigma C(2u^{l} -u^{l-1})\\
y^{l+1}_1 &= \argmin_y\left\{\int_0^T\frac{1}{2}\left\|y-\tilde{y_1}^{l+1}\right\|_2^2 + \frac{\sigma}{2} \left\|y\right\|_2^2 + \sigma\langle y,f\rangle\dif t\right\}\\
y^{l+1}_2 &= \argmin_y\left\{\int_0^T\frac{1}{2}\left\|y-\tilde{y_2}^{l+1}\right\|_2^2 + \sigma\alpha \delta_{B(L^\infty)}(y/\alpha) \dif t\right\}\\
y^{l+1}_3 &= \argmin_y\left\{\int_0^T\frac{1}{2}\left\|y-\tilde{y_3}^{l+1}\right\|_2^2 + \sigma\gamma \delta_{B(L^\infty)}(y/\gamma) \dif t\right\}\\
u^{l+1} &= \argmin_u\left\{\int_0^T \frac{1}{2}\left\|u-(u^l - \tau C^T\boldsymbol{y}^{l+1} )\right\|_2^2 \dif t\right\}
\end{align*}
The subproblem for $y_1$ is a linear $L^2$ problem which has a direct solution. Both problems for $y_2$ and $y_3$ can be solved by projecting point-wise onto the unit ball with radius $\alpha$ respectively $\gamma$. This leads to the iterative scheme
\begin{align*}
\tilde{\boldsymbol{y}}^{l+1} &= \boldsymbol{y}^{l} + \sigma C(2u^{l} - u^{l-1})\\
y^{l+1}_1 &= \frac{\tilde{y}^{l+1}_1 - \sigma f}{\sigma+1}, y^{l+1}_2 = \pi_\alpha(\tilde{y_2}^{l+1}), y^{l+1}_3 = \pi_\gamma(\tilde{y_3}^{l+1})\\
u^{l+1} &= u^l - \tau C^T\boldsymbol{y}^{l+1}
\end{align*}

\paragraph{Problem in $\boldsymbol{v}$:}
The problem in $\boldsymbol{v}$ is a simple $L^1-TV$ optical flow problem. 
As a first step, we define $\lambda := \frac{\beta}{\gamma}$ and split out the regularizer by:
\begin{align*}
F(C\boldsymbol{v}) := \lambda\left\|\nabla \boldsymbol{v}\right\|_1, \text{ with } C\boldsymbol{v} := \begin{pmatrix}\nabla&0\\0&\nabla\end{pmatrix}\boldsymbol{v}.
\end{align*}
We calculate the convex conjugate of $F$ as:
\begin{align*}
F^*(\boldsymbol{y}) =  \lambda \delta_{B(L^\infty)}(\frac{\boldsymbol{y}}{\lambda}).
\end{align*}
Together with the optical flow term we receive the following primal-dual formulation:
\begin{align*}
\argmin_{\boldsymbol{v}}\argmax_{\boldsymbol{y}} \left\|u^{k+1}_t + \nabla u^{k+1} \cdot \boldsymbol{v}\right\|_1 + \left\langle C\boldsymbol{v},\boldsymbol{y} \right\rangle  - \lambda \delta_{B(L^\infty)}(\frac{\boldsymbol{y}}{\lambda}).
\end{align*}
Plugging this into the primal-dual algorithm yields the following problems
\begin{align*}
\boldsymbol{y}^{l+1} &=\argmin_{\boldsymbol{y}}\left\{\frac{1}{2}\left\|\boldsymbol{y}-(\boldsymbol{y}^l + \sigma C(2\boldsymbol{v}^{l} - \boldsymbol{v}^{l-1}))\right\|_2^2 + \lambda\sigma\delta_{B(L^\infty)}(\frac{\boldsymbol{y}}{\lambda}) \right\}\\
\boldsymbol{v}^{l+1} &= \argmin_{\boldsymbol{v}}\left\{ \frac{1}{2}\left\|\boldsymbol{v}- (\boldsymbol{v}^k+\tau C^T \boldsymbol{y}^{l+1})\right\|_2^2 + \tau\left\|u^{k+1}_t+\nabla u^{k+1}\cdot\boldsymbol{v}\right\|_1 \right\}
\end{align*}
Similar to $u$, the subproblem for $\boldsymbol{y}$ can be solved by point-wise projections onto the $L^2$ unit ball with radius $\lambda$. The proximal problem in $\boldsymbol{v}$ can be directly solved by an affine linear shrinkage formula. Therefore, we set
\begin{align*}
\rho(\boldsymbol{v}) := u^{k+1}_t + \nabla u^{k+1}\cdot \boldsymbol{v},\quad \boldsymbol{\beta} := (u^{k+1}_x,u^{k+1}_y).
\end{align*}
Then the solution is given by
\begin{align*}
\boldsymbol{v} = \tilde{\boldsymbol{v}}^{l+1} + \begin{cases}
\tau \boldsymbol{\beta} &\mbox{if } \rho(\tilde{\boldsymbol{v}}^{l+1}) < -\tau \left\|\boldsymbol{\beta}\right\|_2^2 \\  
-\tau \boldsymbol{\beta} &\mbox{if } \rho(\tilde{\boldsymbol{v}}^{l+1}) > \tau \left\|\boldsymbol{\beta}\right\|_2^2 \\
-\frac{\rho(\tilde{\boldsymbol{v}}^{l+1})}{\left\|\boldsymbol{\beta}\right\|_2^2} \boldsymbol{\beta} &\mbox{else }
\end{cases}.
\end{align*}
Combining both formulas we obtain the following scheme:
\begin{align*}
\boldsymbol{y}^{l+1} &= \pi_\lambda( \boldsymbol{y}^l + \sigma C(2\boldsymbol{v}^{l} - \boldsymbol{v}^{l-1}) )\\
\tilde{\boldsymbol{v}}^{l+1} &= \boldsymbol{v}^l-\tau C^T \boldsymbol{y}^{l+1}\\
\boldsymbol{v}^{l+1} &=  \tilde{\boldsymbol{v}}^{l+1} + \begin{cases}
\tau \boldsymbol{\beta} &\mbox{if } \rho(\tilde{\boldsymbol{v}}^{l+1}) < -\tau \left\|\boldsymbol{\beta}\right\|_2^2 \\  
-\tau \boldsymbol{\beta} &\mbox{if } \rho(\tilde{\boldsymbol{v}}^{l+1}) > \tau \left\|\boldsymbol{\beta}\right\|_2^2 \\
-\frac{\rho(\tilde{\boldsymbol{v}}^{l+1})}{\left\|\boldsymbol{\beta}\right\|_2^2} \boldsymbol{\beta} &\mbox{else }
\end{cases}
\end{align*}

\subsection{Discretization}
For the spatial regularization parts $\left\|\nabla u\right\|_1$ and $\left\|\nabla \boldsymbol{v}\right\|_1$ we use forward differences to discretize the involved gradient, respectively backward differences for the adjoint. The coupling term $\left\|u_t+\nabla_x u\cdot\boldsymbol{v}\right\|_1$ is the more challenging part. Using forward differences for the time derivative $u_t$ and central differences for the spatial derivatives $\nabla u$ yields a stable discretization of the transport equation, because the scheme is solved implicitly. Details can be found in Appendix \ref{section:appendixDiscretization} material.\\
As a stopping criterion for both minimization subproblems we use the primal-dual residual (see \cite{goldstein2013adaptive}) as a stopping criterion. For the alternating minimization we measure the difference between two subsequent iterations $k$ and $k+1$ by
\begin{align*}
err_{main} := \frac{\left| u^k-u^{k+1} \right| + \left| \boldsymbol{v}^k-\boldsymbol{v}^{k+1} \right|}{2\left|\Omega\right|}
\end{align*}
and stop if this difference falls below a threshold $\epsilon$. A pseudo code can be found in Appendix \ref{appendixAlgorithm}.

\section{Results}
\label{section:results}
The proposed main algorithm is implemented in MATLAB. For the subproblems in $u$ and $\boldsymbol{v}$ we use the optimization toolbox \textbf{FlexBox} \cite{dirks2016flexible}. The toolbox comes with a C$^{++}$ module, which greatly enhances the runtime. Code and toolbox and be downloaded \cite{dirks2016flexible}. In the following evaluation, the stopping criterion was chosen as $\epsilon=\epsilon_u=\epsilon_v=10^{-6}$. Furthermore, the weighting parameter $\gamma$ for the optical flow constraint in the joint model is set to 1 in all experiments. The evaluated parameter range for  $\alpha$ is the interval $\left[0.01,0.05 \right]$, whereas $\beta$ takes values in $\left[0.05,0.1 \right]$.

\subsection{Joint model versus static optical flow}
First of all we compare our joint TV-TV optical flow model with a TV-L$^1$ optical flow model on the Dimetrodon sequence from \cite{baker2011database} with increasing levels of noise. We use additive Gaussian noise with variance $\sigma\in[0,0.03]$. For the static optical flow model the parameter range [0,0.2]. Both plots from Figure \ref{fig:plotAEEvsNoise} indicate that our joint approach outperformes the static model especially in case of higher noise levels.

\subsection{Denoising and motion estimation}
To demonstrate the benefits of our model we compare the TV-TV optical flow model with different classical methods for image denoising and motion estimation, on different datasets. For image denoising we compare our joint model with a standard 2D ROF model \cite{rudin1992nonlinear} and with a modified 2D+t ROF model from \cite{schaeffer2013real}, which contains an additional time-regularization $\left\| u_t\right\|_1$ for the image sequence. Both regularizers are equally weighted. For motion estimation purposes we calculate the motion field with a TV-L$^1$ optical flow approach for noisy- and for previously TV-denoised image sequences.\\
Due to the limitations of our model to movements of small magnitude that arise from the first order Taylor expansion, we take datasets from the Middlebury optical flow database \cite{baker2011database} and scale down the available ground-truth flow to a maximum magnitude of 1. Afterwards, these downscaled flows are used to create sequences of images by cubic interpolation of $I_1(x+k\boldsymbol{v})$ (k represents the k-th consecutive image). The image sequence is then corrupted with Gaussian noise with variance $\sigma=0.002$. The weights for each algorithm are manually chosen to obtain the corresponding best results. \\
Table $\ref{tab:jointModelResultateUebersicht}$ contains the evaluation results. It becomes clear that our model outperforms both, the standard method for image denoising as well as the method for motion estimation significantly. The visualized results can be found in Figure $\ref{fig:uebersichtJointModelBilder}$. 
\begin{table}
	\centering
	\small
	\begin{tabular}{|c|c|l|l|l|l|l|}
		\hline
		\textbf{Dataset} & \textbf{Algorithm} & \textbf{SSIM} &  \textbf{PSNR} &  \textbf{AEE}&  \textbf{AE}\\
		\hline
		\multirow{4}{*}{Dimetrodon}  & Joint & \textbf{0.970} & \textbf{39.777} & \textbf{0.089} & \textbf{0.050} \\
		\cline{2-6}
		& ROF 2D & 0.949 & 36.862 & - &-  \\
		\cline{2-6}
		& ROF 2D+t & 0.966 & 38.752 & - &-  \\
		\cline{2-6}
		& OF Noisy & - & - & 0.131 & 0.075  \\
		\cline{2-6}
		& OF Denoised & - & - & 0.125 & 0.071  \\
		\hline
		\multirow{4}{*}{Hydrangea} & Joint & \textbf{0.949} & \textbf{37.309} & \textbf{0.031} & \textbf{0.018} \\
		\cline{2-6}
		& ROF 2D & 0.920 & 34.236 & - &- \\
		\cline{2-6}
		& ROF 2D+t & 0.943 & 35.966 & - &- \\
		\cline{2-6}
		& OF Noisy & - & - & 0.073 & 0.045 \\
		\cline{2-6}
		& OF Denoised & - & - & 0.062 & 0.038 \\
		\hline
		\multirow{4}{*}{Rubber Whale} & Joint & \textbf{0.949} & \textbf{37.309} & \textbf{0.031} & \textbf{0.018} \\
		\cline{2-6}
		& ROF 2D & 0.920 & 34.236 & - &- \\
		\cline{2-6}
		& ROF 2D+t & 0.934 & 36.240 & - &- \\
		\cline{2-6}
		& OF Noisy & - & - & 0.072 & 0.045 \\
		\cline{2-6}
		& OF Denoised & - & - & 0.062 & 0.038 \\
		\hline
	\end{tabular}
	\caption{Table comparing the joint motion estimation and image reconstruction model with classical models for image denoising and motion estimation.\textbf{ROF 2D:} Rudin Osher Fatemi model applied to single frames of the image sequence; \textbf{ROF 2D+t:} ROF model with additional time-regularization \textbf{OF Noisy:} TV-L$^1$ optical flow model applied to noisy images; \textbf{OF Denoised:} TV-L$^1$ optical flow model applied to previously TV-denoised images;}
	\label{tab:jointModelResultateUebersicht}
\end{table}

\begin{figure}
	\centering
	\subfloat{\includegraphics[width=.45\linewidth]{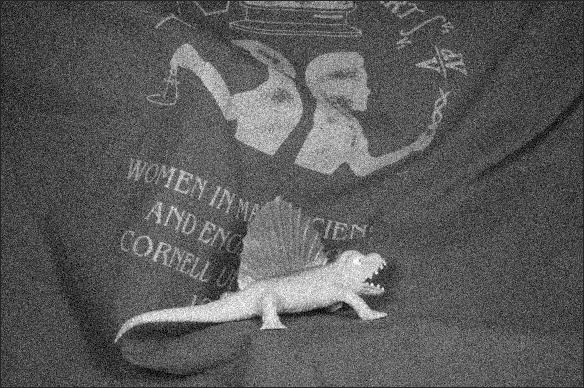}}\quad 
	\subfloat{\includegraphics[width=.45\linewidth]{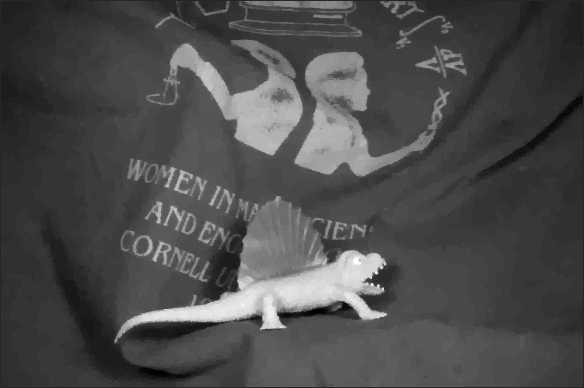}}\\
	\subfloat{\includegraphics[width=.45\linewidth]{flowGT.png}}\quad 
	\subfloat{\includegraphics[width=.45\linewidth]{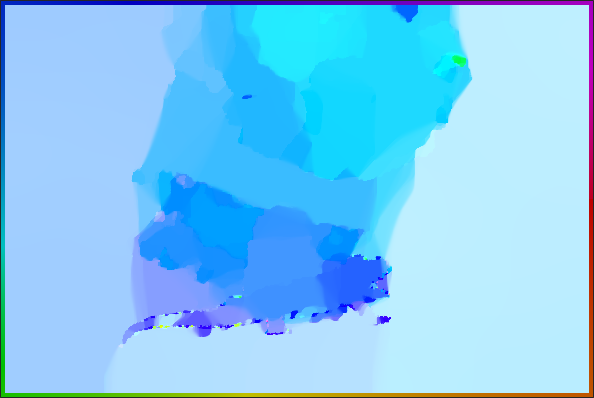}}
	\caption[Input images and corresponding reconstruction]{Top row shows one of the noisy input images and corresponding reconstruction. Bottom row shows ground-truth velocity field and TV-TV optical flow reconstruction. This reconstruction turned out to be the best one compared to other state of the art methods (cp. Table \ref{tab:jointModelResultateUebersicht}).}
	\label{fig:uebersichtJointModelBilder}
\end{figure}

\subsection{Temporal inpainting for real data}
As a real data application we choose the Hamburg Taxi sequence from H.-H. Nagel. Unfortunately, this image sequence has an underlying motion with a magnitude larger than one pixel. The model can be adjusted to this situation by adding additional frames between the known images and using the model to perform temporal inpainting. Therefore, the data fidelity term $\left\|Ku-f\right\|_2^2$ is evaluated only on known frames and the weight $\alpha$ for the total variation is set to zero for unknown frames. The resulting images, time-interpolants and velocity fields are visualized in Figure \ref{fig:temporalInpainting}. We zoom into the image to make differences between original and reconstruction better visible. The complete images can be found in Appendix \ref{section:appendixResults}. In terms of denoising both the car in front and the background are more homogeneous. Moreover, the model generated two time-interpolants and estimated the flow on the whole sequence.
\begin{figure}
	\centering
	\subfloat{\includegraphics[width=.18\linewidth]{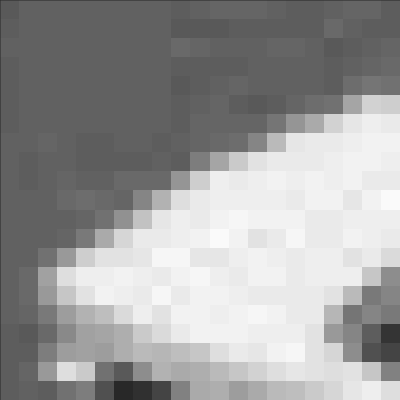}}\enskip
	\subfloat{\includegraphics[width=.18\linewidth]{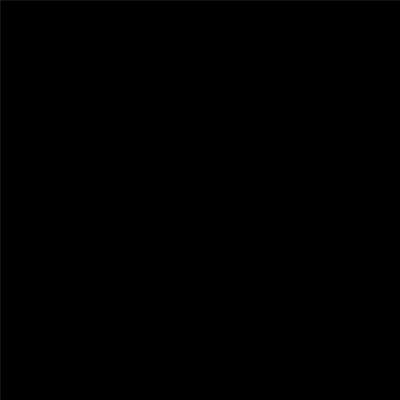}}\enskip
	\subfloat{\includegraphics[width=.18\linewidth]{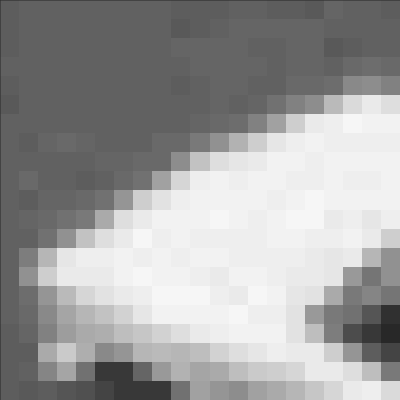}}\enskip
	\subfloat{\includegraphics[width=.18\linewidth]{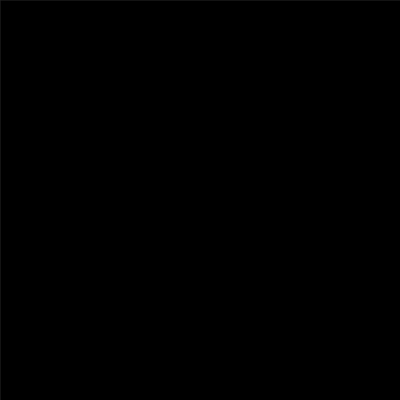}}\enskip
	\subfloat{\includegraphics[width=.18\linewidth]{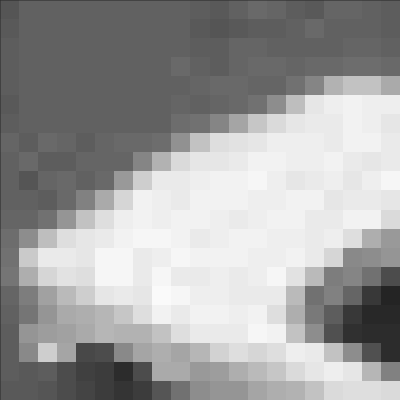}}\\
		
	\subfloat{\includegraphics[width=.18\linewidth]{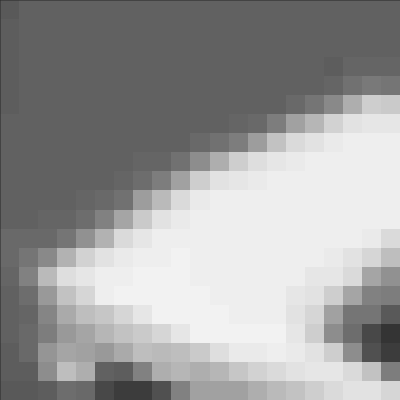}}\enskip
	\subfloat{\includegraphics[width=.18\linewidth]{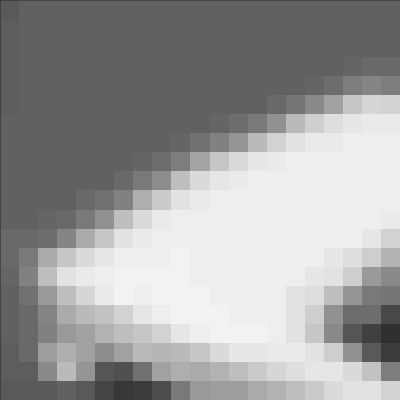}}\enskip
	\subfloat{\includegraphics[width=.18\linewidth]{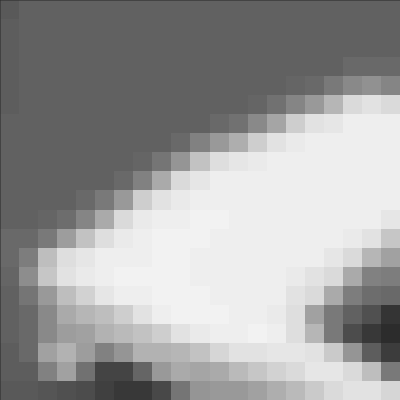}}\enskip
	\subfloat{\includegraphics[width=.18\linewidth]{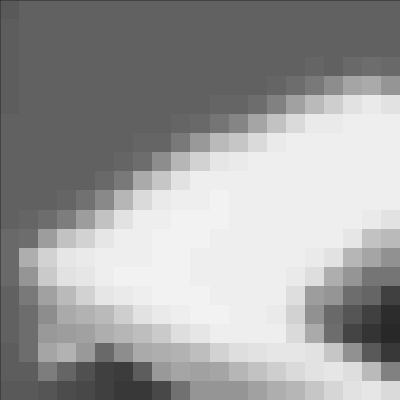}}\enskip
	\subfloat{\includegraphics[width=.18\linewidth]{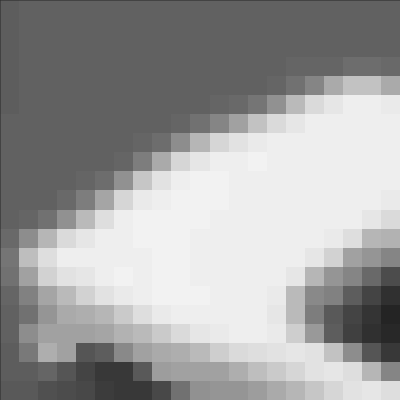}}\\
	
	\subfloat{\includegraphics[width=.18\linewidth]{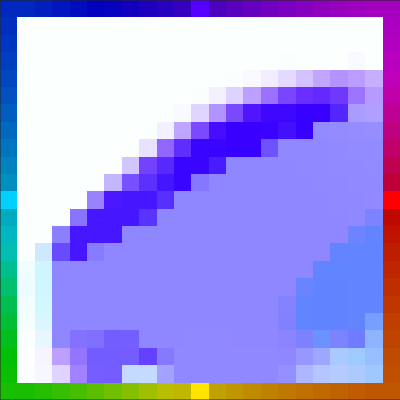}}\enskip
	\subfloat{\includegraphics[width=.18\linewidth]{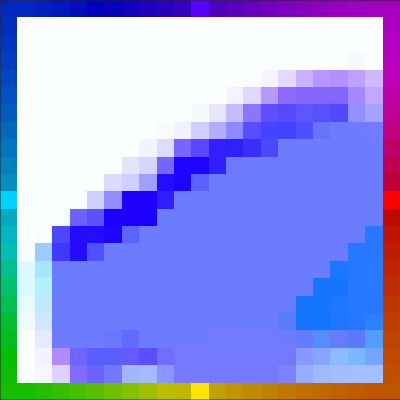}}\enskip
	\subfloat{\includegraphics[width=.18\linewidth]{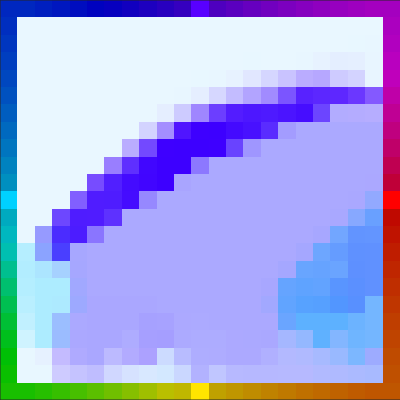}}\enskip
	\subfloat{\includegraphics[width=.18\linewidth]{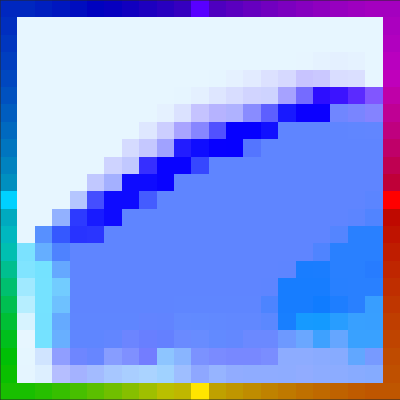}}\enskip
	\subfloat{\includegraphics[width=.18\linewidth]{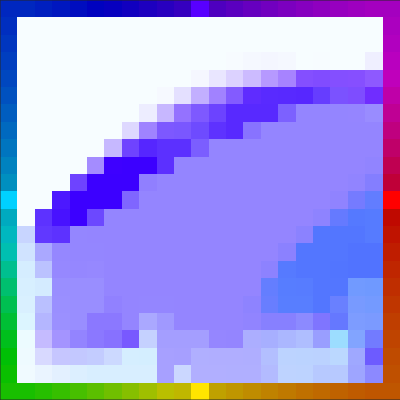}}
	
	\caption[]{Zoom into the Hamburg Taxi sequence (from H.-H. Nagel), complete images can be found in Appendix \ref{section:appendixResults}. Top: Input images, black ones are unknown. Middle: Resulting sequence including time-interpolants. Bottom: Estimated flow fields}
	\label{fig:temporalInpainting}
\end{figure}

\section{Conclusion}
In this paper we propose a joint model for motion estimation and image reconstruction. The model takes a sequence of noisy images as input and simultaneously removes noise while estimating the optical flow between consecutive frames. For the proposed model the existence of a minimizer is proven and we introduce a numerical scheme aiming to solve the variational energy by alternatingly applying a primal-dual method to the problem for the image sequence $u$ and the flow sequence $\boldsymbol{v}$. In the results part we show the benefits of our method in contrast to classical methods for separate image denoising or motion estimation. The presented numerical results included image denoising only but note that the well posedness analysed in this paper holds for general linear operator $K$ and in particular can include image deblurring, image inpainting, sampled 
\section*{Acknowledgments}
This work is based on research that has been done while HD was at the University of Cambridge supported by a David Crighton Fellowship. Moreover, the work has been supported by ERC via Grant EU FP 7 - ERC Consolidator Grant 615216 LifeInverse. MB acknowledges further support by the German Science Foundation DFG via EXC 1003 Cells in Motion Cluster of Excellence, M\"unster, Germany. CBS acknowledges support from Leverhulme Trust project on Breaking the non-convexity barrier, EPSRC grant Nr. EP/M00483X/1 and the EPSRC Centre Nr. EP/N014588/1.

\newpage
\appendix 
\section{Notations}
For this work, the gradient operator $\nabla$ (and the associated divergence operator $\nabla\cdot$) only refers to the spatial dimensions, while time-derivatives are explicitly denoted with subindex $t$.\\
The total variation of $u$ is a semi-norm in the space of functions with bounded variation BV($\Omega$).
\begin{align}
BV(\Omega) := \left\lbrace u \in L^1(\Omega): \left|u\right|_{BV} = \sup_{\phi\in C_0^\infty(\Omega;\R{N}),\left\|\phi\right\|_\infty\leq 1} \int_\Omega u\nabla\cdot\phi dx < \infty \right\rbrace 
\label{def:BV}
\end{align}
For a functional $J:\Omega\rightarrow\mathbb{R}$, the sublevel set is the set of all $u$ for which the functional value lies below $\alpha$,
\begin{align}
\label{def:sublevelset}
	\mathcal{S}_{\alpha} := \left\lbrace u\in\Omega: J(u)< \alpha\right\rbrace.
\end{align}
Due to the fact that the image domain consists of time and space, suitable spaces including space and time are required for the analysis in this paper. The Bochner space
\begin{align*}
L^p(0,T;\CX) := \left\lbrace u : u(\cdot,t)\in\CX \enskip\forall t \in \left[0,T\right], \int_0^T \left\|u(\cdot,t)\right\|_\CX^p dt< \infty\right\rbrace
\end{align*}
is a Banach space with norm
\begin{align*}
\left\|u\right\|_{L^p(0,T;\CX)} = \left(\int_0^T \left\|u(\cdot,t)\right\|_\CX^p dt\right)^\frac{1}{p}.
\end{align*}
A very useful result that will be used in our analysis is the Aubin-Lions Lemma.
\begin{lemma}{Aubin-–Lions}\\
	\label{AubinLionLemma}
	Let $\CX,\CY,\CZ$ be Banach spaces with a compact embedding $\CX\subset\subset\CY$ and a continuous embedding $\CY\hookrightarrow\CZ$. Let $u^n$ be a sequence of bounded functions in $L^p(0,T;\CX)$ and $\partial_t u^n$ be bounded in $L^q(0,T;\CZ)$ (for $q=1$ and $1\leq p<\infty$ or $q>1$ and $1\leq p\leq\infty$). \\
	Then $u^n$ is relatively compact in $L^p(0,T;\CY)$.
\end{lemma}
\begin{proof}
	See $\cite{AubinLion1,AubinLion2,AubinLion3}$.
\end{proof}
In other words: If there exists a compact embedding  from one space into another, the embedding compact carries over to the induced Bochner space if enough time-regularity can be shown.

	\section{Error measures}
	To evaluate the performance of the overall model, quality measures for the reconstructed image sequence and the velocity field are needed.\\
	To measure the quality of the reconstructed image sequence we consider the structural similarity index \textbf{SSIM} \cite{wang2004image}, which measures the difference in luminance, contrast and structure of the ground truth image $u$ and the reconstruction $u_{rec}$ as follows:
	\begin{align*}
	SSIM := \frac{(2\mu_{u}\mu_{u_{rec}} + C_1)(2\sigma_{u,u_{rec}}+C_2)}{(\mu_u^2+\mu_{u_{rec}}^2+C_1)(\sigma_u^2+\sigma_{u_{rec}}^2+C_2)},
	\end{align*}
	where $\mu_u,\mu_{u_{rec}},\sigma_u,\sigma_{u_{rec}}$ and $\sigma_{u,u_{rec}}$ are local means, standard deviations and cross-covariances for ground truth image $u$ and reconstruction $u_{rec}$ respectively. The constants are fixed to $C_1=0.01^2$ and $C_2=0.03^2$. The SSIM takes values between $-1$ and $1$, where $1$ stands for perfect similarity. Moreover, we calculate the signal-to-noise ratio \textbf{SNR} and peak signal-to-noise ratio \textbf{PSNR} between ground-truth $u$ and reconstruction $u_{rec}$. 
	\begin{align*}
	SNR := 10\log_{10}(\frac{\text{mean}(u^2)}{\text{mean}((u-u_{rec})^2)} ),\\
	PSNR := 10\log_{10}(\frac{\text{max}(u^2)}{\text{mean}((u-u_{rec})^2)} ).
	\end{align*}
	For the evaluation of the motion field we refer to the work of Baker et. al. \cite{baker2011database}. The most intuitive measure presented there is the average endpoint error \textbf{AEE} proposed in \cite{otte1994optical}, which is the vector-wise Euclidean norm of the difference vector $\boldsymbol{v}-\boldsymbol{v}_{GT}$, where $\boldsymbol{v}$ is the reconstructed velocity field and $\boldsymbol{v}_{GT}$ is the true velocity field. For normalization, the difference is divided by $\left|\Omega\right|$ and we have
	\begin{align}
	\label{definition:aee}
	AEE := \frac{1}{\left|\Omega\right|}\int_{\Omega}\sqrt{(v^1(x)-v^1_{GT}(x))^2 + (v^2(x)-v^2_{GT}(x))^2} \dif x.
	\end{align}
	Another measure we use is the angular error \textbf{AE}, which goes back to the work of Fleet and Jepson \cite{fleet1990computation} and a survey of Barron \textit{et al.} \cite{barron1994performance}. Here $\boldsymbol{v}$ and $\boldsymbol{v}_{GT}$ are projected into the 3-D space (to avoid division by zero) and normalized as follows
	\begin{align*}
	\hat{\boldsymbol{v}} := \frac{(v^1,v^2,1)}{\sqrt{\left\|\boldsymbol{v}\right\|^2 + 1}}, \quad \quad \hat{\boldsymbol{v}}_{GT} := \frac{(v_{GT}^1,v^2_{GT},1)}{\sqrt{\left\|\boldsymbol{v}_{GT}\right\|^2 + 1}}.
	\end{align*}
	The error is then calculated measuring the angle between $\hat{\boldsymbol{v}}$ and $\hat{\boldsymbol{v}}_{GT}$ in the continuous setting as
	\begin{align*}
	AE := \frac{1}{\left|\Omega\right|}\int_{\Omega} \arccos(\hat{\boldsymbol{v}}(x)\cdot\hat{\boldsymbol{v}}_{GT}(x))\dif x,
	\end{align*}

	\section{Discretization}
	\label{section:appendixDiscretization}
	First, we assume the underlying space-time grid to consist of the following set of discrete points:
	\begin{align*} 
	\left\lbrace (i,j,t) : i=0,\ldots,n_x,j=0,\ldots,n_y,t=0,\ldots,n_t \right\rbrace
	\end{align*}
	The resulting discrete derivatives for $v^i$ are calculated using forward differences and Neumann boundary conditions. The corresponding adjoint operator consists of backward differences with Dirichlet boundary conditions and is applied to the dual variables $\boldsymbol{y}$. The resulting scheme reads:
	\begin{align*}
	v^i_x(i,j) &= \begin{cases} v(i+1,j)-v(i,j) &\mbox{if } i<n_x \\ 
	0 & \mbox{if } i=n_x \end{cases}\\
	v^i_y(i,j) &= \begin{cases} v(i,j+1)-v(i,j) &\mbox{if } j<n_y \\ 
	0 & \mbox{if } j=n_y \end{cases}\\
	\nabla\cdot \boldsymbol{y}(i,j) &= 
	\begin{cases} 
	y_1(i,j)-y_1(i-1,j) &\mbox{if } i>0 \\ 
	y_1(i,j) & \mbox{if } i=0\\
	-y_1(i-1,j) & \mbox{if } i=n_x 
	\end{cases}\\
	&+ 
	\begin{cases} 
	y_2(i,j)-y_2(i,j-1) &\mbox{if } j>0 \\ 
	y_2(i,j) & \mbox{if } j=0 \\
	-y_2(i,j-1) & \mbox{if } j=n_y .
	\end{cases}
	\end{align*} 
	The discrete derivatives for the regularizer of $u$ have the same structure. For the operator in the optical flow part we use a forward discretization for the temporal derivative and a central discretization for the spatial derivative. Again, Neumann boundary conditions are applied:
	\begin{align*}
	u_t(i,j,t) &= \begin{cases} u(i,j,t+1)-u(i,j,t) &\mbox{if } t<n_t \\ 
	0 & \mbox{else } \end{cases},\\
	u_x(i,j,t) &= \begin{cases} \frac{u(i+1,j,t)-u(i-1,j,t)}{2} &\mbox{if } i>0 \mbox{ and } i<n_x \mbox{ and } t<n_t \\ 
	0 & \mbox{else } \end{cases},\\
	u_y(i,j,t) &= \begin{cases} \frac{u(i,j+1,t)-u(i,j-1,t)}{2} &\mbox{if } j>0 \mbox{ and } j<n_y \mbox{ and } t<n_t \\ 
	0 & \mbox{else } \end{cases}.
	\end{align*}
	The adjoint operator then yields:
	\begin{align*}
	y_t(i,j,t) &= \begin{cases} y(i,j,t)-y(i,j,t-1) &\mbox{if } t>0 \mbox{ and } t<n_t\\ 
	y(i,j,t) & \mbox{if } t=0\\
	-y(i,j,t-1) & \mbox{if } t=n_t \end{cases},\\
	y_x(i,j,t) &= \begin{cases} \frac{y(i+1,j,t)-y(i-1,j,t)}{2} &\mbox{if } i>1 \mbox{ and } i<n_x-1 \mbox{ and } t<n_t \\ 
	\frac{y(i+1,j,t)}{2} &\mbox{if } i\leq 1 \mbox{ and } t<n_t\\
	\frac{y(i-1,j,t)}{2} &\mbox{if } i\geq n_x-1 \mbox{ and } t<n_t\\
	0&\mbox{else}\end{cases},\\
	y_y(i,j,t) &= \begin{cases} \frac{y(i,j+1,t)-y(i,j-1,t)}{2} &\mbox{if } j>1 \mbox{ and } j<n_y-1 \mbox{ and } t<n_t \\ 
	\frac{y(i,j+1,t)}{2} &\mbox{if } j\leq 1 \mbox{ and } t<n_t\\
	\frac{y(i,j-1,t)}{2} &\mbox{if } j\geq n_y-1 \mbox{ and } t<n_t\\
	0&\mbox{else}\end{cases}.\\
	\end{align*}
	
	\newpage
	\section{Algorithm}
	\label{appendixAlgorithm}
	\begin{algorithm}[H]
		\caption{Joint $TV-TV$ Optical Flow Motion Estimation and Image Reconstruction}
		\label{algorithmTVTV}
		\begin{algorithmic}
			\Statex
			\Function{JointTVTVOpticalFlow}{$f,\alpha,\beta,\gamma,K$}
			\Let{$v,u$}{$0$}
			\While{$\epsilon<threshold$}
			\Let{$u_{OldM}$}{$u$}
			\Let{$\boldsymbol{v}_{OldM}$}{$\boldsymbol{v}$}
			\Let{$y,\bar{u}$}{$0$}
			\While{$\epsilon_u<threshold$}
			\Let{$u_{Old}$}{$u$}
			\Let{$\tilde{\boldsymbol{y}}$}{$\boldsymbol{y}+\sigma C_u\bar{u}$}
			\Let{$y_1$}{$\frac{\tilde{y_1}-\sigma f}{\sigma+1}$}
			\Let{$y_2$}{$\pi_\alpha(\tilde{y_2})$}
			\Let{$y_3$}{$\pi_\gamma(\tilde{y_3})$}
			\Let{$u$}{$u-\tau C^T_u\boldsymbol{y}$}
			\Let{$\bar{u}$}{$2u-u_{Old}$}
			\EndWhile
			\Let{$y,\bar{\boldsymbol{v}}$}{$0$}
			\While{$\epsilon_v<threshold$}
			\Let{$\boldsymbol{v}_{Old}$}{$\boldsymbol{v}$}
			\Let{$\tilde{\boldsymbol{y}}$}{$\boldsymbol{y}+\sigma C_v\bar{\boldsymbol{v}}$}
			\Let{$\boldsymbol{y}$}{$\pi_\lambda(\tilde{\boldsymbol{y}})$}
			\Let{$\tilde{\boldsymbol{v}}$}{$\boldsymbol{v}-\tau C^T_v\boldsymbol{y}$}
			\Let{$\boldsymbol{v}$}{$solveAffine(\boldsymbol{\tilde{v}})$}
			\Let{$\bar{\boldsymbol{v}}$}{$2\boldsymbol{v}-\boldsymbol{v}_{Old}$}
			\EndWhile
			\Let{$\epsilon$}{$\frac{\left| u-u_{OldM} \right| + \left| \boldsymbol{v}-\boldsymbol{v}_{OldM} \right|}{2\left|\Omega\right|}$}
			\EndWhile
			\State \Return{$v$}
			\EndFunction
		\end{algorithmic}
	\end{algorithm}
	
	\newpage
	\section{Results}
	\label{section:appendixResults}
	\begin{figure}[H]
		\centering
		\subfloat{\includegraphics[width=.32\linewidth]{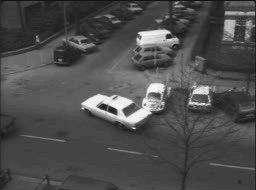}}\enskip
		\subfloat{\includegraphics[width=.32\linewidth]{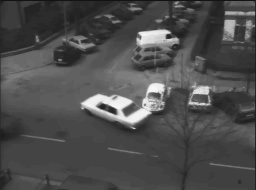}}\enskip
		\subfloat{\includegraphics[width=.32\linewidth]{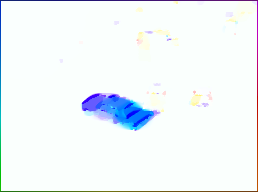}}\\
		\subfloat{\includegraphics[width=.32\linewidth]{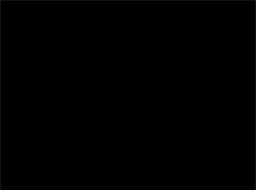}}\enskip
		\subfloat{\includegraphics[width=.32\linewidth]{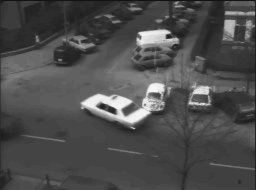}}\enskip
		\subfloat{\includegraphics[width=.32\linewidth]{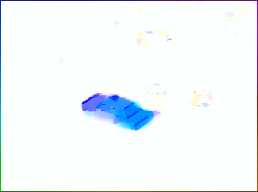}}\\
		\subfloat{\includegraphics[width=.32\linewidth]{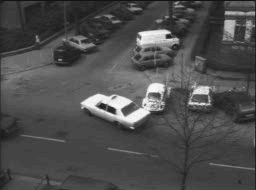}}\enskip
		\subfloat{\includegraphics[width=.32\linewidth]{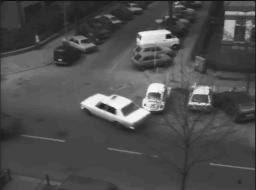}}\enskip
		\subfloat{\includegraphics[width=.32\linewidth]{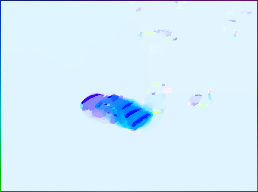}}\\
		\subfloat{\includegraphics[width=.32\linewidth]{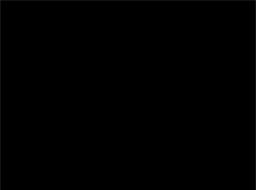}}\enskip
		\subfloat{\includegraphics[width=.32\linewidth]{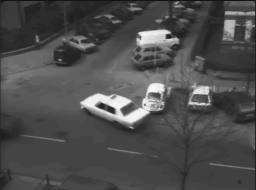}}\enskip
		\subfloat{\includegraphics[width=.32\linewidth]{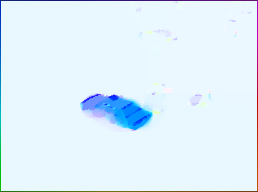}}
		\caption[]{Zoom into the Hamburg Taxi sequence (from H.-H. Nagel). Left: Input images, black are unknown. Middle: Resulting sequence including time-interpolants. Right: Estimated flow fields}
		\label{fig:temporalInpaintingComplete}
	\end{figure}
	
	\newpage
	
\bibliographystyle{plain}
\bibliography{citations}

\begin{thebibliography}{10}

\bibitem{aubert1999computing}
Gilles Aubert, Rachid Deriche, and Pierre Kornprobst.
\newblock Computing optical flow via variational techniques.
\newblock {\em SIAM Journal on Applied Mathematics}, 60(1):156--182, 1999.

\bibitem{AubinLion1}
Jean-Pierre Aubin.
\newblock Un th\'eor\`eme de compacit\'e.
\newblock {\em C. R. Acad. Sci. Paris}, 256:5042--5044, 1963.

\bibitem{baker2011database}
Simon Baker, Daniel Scharstein, JP~Lewis, Stefan Roth, Michael~J Black, and
  Richard Szeliski.
\newblock A database and evaluation methodology for optical flow.
\newblock {\em International Journal of Computer Vision}, 92(1):1--31, 2011.

\bibitem{bar2007variational}
Leah Bar, Benjamin Berkels, Martin Rumpf, and Guillermo Sapiro.
\newblock A variational framework for simultaneous motion estimation and
  restoration of motion-blurred video.
\newblock In {\em Computer Vision, 2007. ICCV 2007. IEEE 11th International
  Conference on}, pages 1--8. IEEE, 2007.

\bibitem{barron1994performance}
John~L Barron, David~J Fleet, and Steven~S Beauchemin.
\newblock Performance of optical flow techniques.
\newblock {\em International journal of computer vision}, 12(1):43--77, 1994.

\bibitem{benamou2000computational}
Jean-David Benamou and Yann Brenier.
\newblock A computational fluid mechanics solution to the {Monge-Kantorovich}
  mass transfer problem.
\newblock {\em Numerische Mathematik}, 84(3):375--393, 2000.

\bibitem{benamou2002monge}
Jean-David Benamou, Yann Brenier, and Kevin Guittet.
\newblock The {Monge--Kantorovitch} mass transfer and its computational fluid
  mechanics formulation.
\newblock {\em International Journal for Numerical methods in fluids},
  40(1-2):21--30, 2002.

\bibitem{benamou:hal-01295299}
Jean-David Benamou, Guillaume Carlier, and Filippo Santambrogio.
\newblock {Variational Mean Field Games}.
\newblock working paper or preprint, March 2016.

\bibitem{benning2013higher}
Martin Benning, Christoph Brune, Martin Burger, and Jahn M{\"u}ller.
\newblock Higher-order {TV} methods—enhancement via {Bregman} iteration.
\newblock {\em Journal of Scientific Computing}, 54(2-3):269--310, 2013.

\bibitem{borzi2003optimal}
Alfio Borzi, Kazufumi Ito, and Karl Kunisch.
\newblock Optimal control formulation for determining optical flow.
\newblock {\em SIAM journal on scientific computing}, 24(3):818--847, 2003.

\bibitem{bredies2014regularization}
Kristian Bredies and Martin Holler.
\newblock Regularization of linear inverse problems with total generalized
  variation.
\newblock {\em Journal of Inverse and Ill-posed Problems}, 22(6):871--913,
  2014.

\bibitem{bredies2010total}
Kristian Bredies, Karl Kunisch, and Thomas Pock.
\newblock Total generalized variation.
\newblock {\em SIAM Journal on Imaging Sciences}, 3(3):492--526, 2010.

\bibitem{brune20104d}
Christoph Brune.
\newblock {\em 4D imaging in tomography and optical nanoscopy}.
\newblock PhD thesis, PhD thesis, University of M{\"u}nster, Germany, 2010.

\bibitem{burgerlevel2008}
Martin Burger, Andrea~CG Mennucci, Stanley Osher, and Martin Rumpf.
\newblock {\em Level Set and PDE Based Reconstruction Methods in Imaging}.
\newblock Springer, 2008.

\bibitem{suhr}
Martin Burger, Jan Modersitzki, and Sebastian Suhr.
\newblock A nonlinear variational approach to motion-corrected reconstruction
  of density images.
\newblock {\em arXiv preprint arXiv:1511.09048}, 2015.

\bibitem{cengiz1998dual}
BAHAETTIN Cengiz.
\newblock The dual of the {B}ochner space ${L}^p(\mu,{E})$ for arbitrary $\mu$.
\newblock {\em Turkish Journal of Mathematics}, 22:343--348, 1998.

\bibitem{chambolle2011first}
Antonin Chambolle and Thomas Pock.
\newblock A first-order primal-dual algorithm for convex problems with
  applications to imaging.
\newblock {\em Journal of Mathematical Imaging and Vision}, 40(1):120--145,
  2011.

\bibitem{dirks}
Hendrik Dirks.
\newblock {\em Variational Methods for Joint Motion Estimation and Image
  Reconstruction}.
\newblock PhD thesis, WWU M\"unster, 2015.

\bibitem{dirks2016flexible}
Hendrik Dirks.
\newblock A flexible primal-dual toolbox.
\newblock {\em arXiv preprint}, 2016.
\newblock http://www.flexbox.im.

\bibitem{fleet1990computation}
David~J Fleet and Allan~D Jepson.
\newblock Computation of component image velocity from local phase information.
\newblock {\em International Journal of Computer Vision}, 5(1):77--104, 1990.

\bibitem{gilland2002simultaneous}
David~R Gilland, Bernard~A Mair, James~E Bowsher, and Ronald~J Jaszczak.
\newblock Simultaneous reconstruction and motion estimation for gated cardiac
  ect.
\newblock {\em Nuclear Science, IEEE Transactions on}, 49(5):2344--2349, 2002.

\bibitem{goldstein2013adaptive}
Tom Goldstein, Ernie Esser, and Richard Baraniuk.
\newblock Adaptive primal-dual hybrid gradient methods for saddle-point
  problems.
\newblock {\em arXiv preprint arXiv:1305.0546}, 2013.

\bibitem{horn1981determining}
Berthold~K Horn and Brian~G Schunck.
\newblock Determining optical flow.
\newblock In {\em 1981 Technical Symposium East}, pages 319--331. International
  Society for Optics and Photonics, 1981.

\bibitem{kosters2011emrecon}
Thomas K{\"o}sters, Klaus~P Sch{\"a}fers, and Frank W{\"u}bbeling.
\newblock Emrecon: An expectation maximization based image reconstruction
  framework for emission tomography data.
\newblock In {\em Nuclear Science Symposium and Medical Imaging Conference
  (NSS/MIC), 2011 IEEE}, pages 4365--4368. IEEE, 2011.

\bibitem{AubinLion2}
J.-L. Lions.
\newblock {\em Quelques m\'ethodes de r\'esolution des probl\`emes aux limites
  non lin\'eaires}.
\newblock Dunod; Gauthier-Villars, Paris, 1969.

\bibitem{meyers1977integral}
Norman~G Meyers and William~P Ziemer.
\newblock Integral inequalities of poincare and wirtinger type for bv
  functions.
\newblock {\em American Journal of Mathematics}, pages 1345--1360, 1977.

\bibitem{mitzel2009video}
Dennis Mitzel, Thomas Pock, Thomas Schoenemann, and Daniel Cremers.
\newblock Video super resolution using duality based {TV-L1} optical flow.
\newblock In {\em Pattern Recognition}, pages 432--441. Springer, 2009.

\bibitem{otte1994optical}
Michael Otte and H-H Nagel.
\newblock Optical flow estimation: advances and comparisons.
\newblock In {\em Computer Vision�ECCV'94}, pages 49--60. Springer, 1994.

\bibitem{papenberg2006highly}
Nils Papenberg, Andr{\'e}s Bruhn, Thomas Brox, Stephan Didas, and Joachim
  Weickert.
\newblock Highly accurate optic flow computation with theoretically justified
  warping.
\newblock {\em International Journal of Computer Vision}, 67(2):141--158, 2006.

\bibitem{pock2009algorithm}
Thomas Pock, Daniel Cremers, Horst Bischof, and Antonin Chambolle.
\newblock An algorithm for minimizing the {Mumford-Shah} functional.
\newblock In {\em Computer Vision, 2009 IEEE 12th International Conference on},
  pages 1133--1140. IEEE, 2009.

\bibitem{rudin1992nonlinear}
Leonid~I Rudin, Stanley Osher, and Emad Fatemi.
\newblock Nonlinear total variation based noise removal algorithms.
\newblock {\em Physica D: Nonlinear Phenomena}, 60(1):259--268, 1992.

\bibitem{rudin1973functional}
Walter Rudin.
\newblock {\em Functional analysis, 1973}.
\newblock McGraw-Hill, New York, 1973.

\bibitem{sawatzky2008accurate}
Alex Sawatzky, Christoph Brune, Frank Wubbeling, Thomas Kosters, Klaus
  Schafers, and Martin Burger.
\newblock Accurate {EM-TV} algorithm in pet with low {SNR}.
\newblock In {\em 2008 IEEE Nuclear Science Symposium Conference Record}, pages
  5133--5137. IEEE, 2008.

\bibitem{schaeffer2013real}
Hayden Schaeffer, Yi~Yang, and Stanley Osher.
\newblock Real-time adaptive video compressive sensing.
\newblock {\em UCLA CAM, Tech. Reports}, 2013.

\bibitem{shen2007map}
Huanfeng Shen, Liangpei Zhang, Bo~Huang, and Pingxiang Li.
\newblock A map approach for joint motion estimation, segmentation, and super
  resolution.
\newblock {\em Image Processing, IEEE Transactions on}, 16(2):479--490, 2007.

\bibitem{shen2002mathematical}
Jianhong Shen and Tony~F Chan.
\newblock Mathematical models for local nontexture inpaintings.
\newblock {\em SIAM Journal on Applied Mathematics}, 62(3):1019--1043, 2002.

\bibitem{AubinLion3}
Jacques Simon.
\newblock Compact sets in the space ${L}^p (0,{T};{B})$.
\newblock {\em Annali di Matematica pura ed applicata}, 146(1):65--96, 1986.

\bibitem{tartar1983compensated}
Luc Tartar.
\newblock The compensated compactness method applied to systems of conservation
  laws.
\newblock In {\em Systems of nonlinear partial differential equations}, pages
  263--285. Springer, 1983.

\bibitem{tomasi1992shape}
Carlo Tomasi and Takeo Kanade.
\newblock Shape and motion from image streams under orthography: a
  factorization method.
\newblock {\em International Journal of Computer Vision}, 9(2):137--154, 1992.

\bibitem{unger2010convex}
Markus Unger, Thomas Pock, Manuel Werlberger, and Horst Bischof.
\newblock A convex approach for variational super-resolution.
\newblock In {\em Pattern Recognition}, pages 313--322. Springer, 2010.

\bibitem{Wang07afast}
Yilun Wang, Wotao Yin, and Yin Zhang.
\newblock A fast algorithm for image deblurring with total variation
  regularization, 2007.

\bibitem{wang2004image}
Zhou Wang, Alan~C Bovik, Hamid~R Sheikh, and Eero~P Simoncelli.
\newblock Image quality assessment: from error visibility to structural
  similarity.
\newblock {\em Image Processing, IEEE Transactions on}, 13(4):600--612, 2004.

\bibitem{zach2007duality}
Christopher Zach, Thomas Pock, and Horst Bischof.
\newblock A duality based approach for realtime {TV-L1} optical flow.
\newblock In {\em Pattern Recognition}, pages 214--223. Springer, 2007.

\end{thebibliography}
\end{document}